\documentclass[11pt]{amsart}
\usepackage{geometry}                
\usepackage{graphicx,amssymb,epstopdf,tikz,url,enumerate}
\usepackage{mathabx} 

\usetikzlibrary{arrows,chains,matrix,positioning,scopes}
\makeatletter
\tikzset{join/.code=\tikzset{after node path={%
\ifx\tikzchainprevious\pgfutil@empty\else(\tikzchainprevious)%
edge[every join]#1(\tikzchaincurrent)\fi}}}
\makeatother
\tikzset{>=stealth',every on chain/.append style={join},
         every join/.style={->}}
\tikzstyle{labeled}=[execute at begin node=$\scriptstyle,
   execute at end node=$]

\theoremstyle{plain}
\newtheorem{theorem}{Theorem}[section]
\newtheorem{lemma}[theorem]{Lemma}
\newtheorem{corollary}[theorem]{Corollary}
\newtheorem{proposition}[theorem]{Proposition}

\theoremstyle{definition}
\newtheorem{definition}[theorem]{Definition}

\newtheorem{example}[theorem]{Example}
\newtheorem{conjecture}[theorem]{Conjecture}

\newtheorem{remark}[theorem]{Remark}

\theoremstyle{remark}

\numberwithin{equation}{section}

\newcommand{\row}{\rightarrow}

\renewcommand{\phi}{\varphi}
\newcommand{\lam}{\lambda}
\renewcommand{\epsilon}{\varepsilon}

\renewcommand{\P}{P}
\newcommand{\Q}{Q}
\newcommand{\R}{R}

\newcommand{\calI}{\mathcal{I}} %
\newcommand{\calZ}{\mathfrak{Z}} 
\newcommand{\ZZ}{\mathbb{Z}}

\newcommand{\N}{\mathsf{N}} 
\newcommand{\E}{\mathsf{E}} 

\newcommand{\dinv}{\ensuremath{\mathsf{dinv}}}
\newcommand{\bounce}{\ensuremath{\mathsf{bounce}}}
\newcommand{\area}{\ensuremath{\mathsf{area}}}

\renewcommand{\sl}{\ensuremath{\mathsf{sl}}}

\newcommand{\rk}{\ensuremath{\mathsf{rk}}}

\newcommand{\core}{\mathfrak{c}}
\newcommand{\overbar}[1]{\mkern 1.5mu\overline{\mkern-1.5mu#1\mkern-1.5mu}\mkern 1.5mu}
\renewcommand{\bar}{\overbar}

\newcommand{\DD}{\mathfrak{D}} 

\newcommand{\len}{l}

\newcommand{\flip}{\mathrm{flip}}

\usepackage{ifthen}
\newboolean{draft}
\setboolean{draft}{true}
\ifdraft
\newcommand{\TODO}[2][To do: ]{\textcolor{red}{\textbf{#1#2}}}
\else
\newcommand{\TODO}[2][]{}
\fi

\usepackage{todonotes}

\definecolor{darkblue}{rgb}{0, 0, .4}
\definecolor{grey}{rgb}{.7, .7, .7}
  \usepackage[breaklinks]{hyperref}
  \hypersetup{
            colorlinks=true,
            linkcolor=darkblue,
            anchorcolor=darkblue,
            citecolor=darkblue,
            pagecolor=darkblue,
            urlcolor=darkblue,
            pdftitle={},
            pdfauthor={}
  }

\graphicspath{{figures/}}  

\title{Combinatorics of the zeta map on rational Dyck paths}

\author{Cesar Ceballos}
\address{Faculty of Mathematics, University of Vienna, Oskar-Morgenstern-Platz 1, 1090 Vienna, Austria}
\email{\href{mailto:cesar.ceballos@univie.ac.at}{\texttt{cesar.ceballos@univie.ac.at}}}
\urladdr{\url{http://garsia.math.yorku.ca/~ceballos/}}

\author{Tom Denton}
\address{Department of Mathematics and Statistics\\ York University\\ Toronto\\ Ontario M3J 1P3\\ Canada}
\email{\href{mailto:sdenton4@gmail.com}{\texttt{sdenton4@gmail.com}}}
\urladdr{\url{http://inventingsituations.net/}}

\author{Christopher R.\ H.\ Hanusa}
\address{Department of Mathematics \\ Queens College (CUNY) \\ 65-30 Kissena Blvd. \\ Flushing, NY 11367\\ United States}
\email{\href{mailto:chanusa@qc.cuny.edu}{\texttt{chanusa@qc.cuny.edu}}}
\urladdr{\url{http://qc.edu/~chanusa/}}

\thanks{
The first author was supported by the government of Canada through a Banting Postdoctoral Fellowship. He was also supported by a York University research grant.
The second author was supported by a postdoctoral fellowship with York
University and the Fields Institute, with additional support from
NSERC.
The third author gratefully acknowledges support from PSC-CUNY Research Awards TRADA-44-168 and TRADA-45-60.
}

\keywords{Dyck path, rational Dyck path, lattice path, core partition, zeta map, eta map, sweep map, lasers, conjugate-area map, area statistic, dinv statistic}

\begin{document}

\begin{abstract}
An $(a,b)$-Dyck path $\P$ is a lattice path from $(0,0)$ to $(b,a)$ that stays above the line $y=\frac{a}{b}x$.  The zeta map is a curious rule that maps the set of $(a,b)$-Dyck paths into itself; it is conjecturally bijective, and we provide progress towards proof of bijectivity in this paper, by showing that knowing zeta of $\P$ and zeta of $\P$ conjugate is enough to recover $\P$. 

Our method begets an area-preserving involution $\chi$ on the set of $(a,b)$-Dyck paths when $\zeta$ is a bijection, as well as a new method for calculating $\zeta^{-1}$ on classical Dyck paths.  For certain nice $(a,b)$-Dyck paths we give an explicit formula for $\zeta^{-1}$ and $\chi$ and for additional $(a,b)$-Dyck paths we discuss how to compute $\zeta^{-1}$ and $\chi$ inductively.

We also explore Armstrong's skew length statistic and present two new combinatorial methods for calculating the zeta map involving lasers and interval intersections.  We provide a combinatorial statistic $\delta$ that can be used to recursively compute $\zeta^{-1}$ and show that $\delta$ is computable from $\zeta(P)$ in the Fuss-Catalan case.
\end{abstract}

\maketitle

\section{Introduction}
\label{sec.intro}

Let $a$ and $b$ be relatively prime positive integers and let $\DD_{a,b}$ be the set of $(a,b)$-Dyck paths, lattice paths $\P$ from $(0,0)$ to $(b,a)$ staying above the line $y=\frac{a}{b}x$.  These paths are often called {\em rational Dyck paths} and they generalize the classical and well-studied {\em Dyck paths}.

\renewcommand*{\thefootnote}{\fnsymbol{footnote}}

We study a remarkable function $\zeta$ on rational Dyck paths conjectured to be an automorphism\footnote{After this article was accepted for publication, we learned that Nathan Williams proved that the zeta map and its sweep map brethren are indeed bijective using other methods. \cite{Williams}}, which has received considerable attention lately; this ``zeta map'' generalizes the map on standard Dyck paths discovered by Haiman in the study of diagonal harmonics and $q,t$-Catalan numbers~\cite{haglund2008q}.  Combinatorial definitions of $q,t$-statistics for classical Dyck paths were famously difficult to find, but were nearly simultaneously discovered by Haglund and Haiman.  Interestingly, they discovered two \emph{different} pairs of statistics: Haiman found $\area$ and $\dinv$ shortly after Haglund discovered $\bounce$ and $\area$ statistics.  The zeta map was then uncovered, which satisfies $\bounce(\zeta(\P))=\area(\P)$ and $\area(\zeta(\P))=\dinv(\P)$.

Many details about the zeta map have been gathered and unified in a comprehensive article by Armstrong, Loehr, and Warrington \cite{ALW-sweep}, including progress on proving its bijectivity in certain cases such as $(a,am\pm 1)$-Dyck paths \cite{Loehr,GMII} (which is associated to the Fuss-Catalan numbers).  The zeta map was shown to be a bijection in these special cases by way of a ``bounce path'' by which zeta inverse could be computed.  However, constructing such a bounce path for the general $(a,b)$ case remains elusive.  Armstrong, Loehr, and Warrington showed that there is a much larger family of sweep maps (for which the zeta map is a special case) which extensive computational exploration suggests are also bijective.  A construct of theirs upon which we have relied heavily is the notion of the \emph{levels} of a lattice path.

Recent progress related to rational Dyck paths has been made in the case when $a\leq3$ by Gorsky and Mazin and by Kaliszewski and Li \cite{GMII, KL14}, when $a=4$ by Lee, Li, and Loehr \cite{LLL14} in connection with the $q,t$-symmetry of the rational Catalan numbers.  A type $C$ analog of the zeta map has been introduced by Sulzgruber and Thiel \cite{ST14}.  Rational Dyck paths also are intimately entwined in the study of rational parking functions and MacDonald polynomials, with recent work by Gorsky, Mazin, and Vazirani \cite{GMV} and when $a$ and $b$ are not relatively prime by Bergeron, Garsia, Levin, and Xin~\cite{BGLX}.

Our goal is to explore the following conjecture:

\begin{conjecture}[\cite{ALW-sweep,GMI}]
\label{conj.main}
Let $a$ and $b$ be relatively prime positive integers.  The {\em zeta map} $\zeta:\DD_{a,b}\rightarrow\DD_{a,b}$ is a bijection.
\end{conjecture}

Our perspective is that there are in fact two maps, the zeta map and the eta map, which jointly contain enough information to recover the original path.  In Section~\ref{sec:zetaeta}, we provide a straightforward algorithm for recovering $\P$ from the combined data of $Q=\zeta(\P)$ and $R=\eta(\P)$.  What we find interesting is that the information contained solely in $\zeta(\P)$ does not seem to be enough to reconstruct $\P$ directly.  Our argument does not give an explicit construction of $\zeta^{-1}(Q)$, nor do we construct a bounce path.  

The zeta and eta maps appeared previously in the work of Gorsky and Mazin (See $G_{n,m}$ and~$G_{m,n}$ in \cite{GMII}) and in the work of Armstrong, Loehr, and Warrington (varying the direction of the sweep map in \cite{ALW-sweep}). Although, they were never used simultaneously as we do in this paper.  
The eta map is based on a natural notion of conjugation on rational Dyck paths explored in Section~\ref{sec:conjugate} that arises from Anderson's bijection \cite{And02} between $(a,b)$-Dyck paths and simultaneous $(a,b)$-core partitions, which in turn are related to many more combinatorial interpretations. (See \cite{AHJ14} for additional background.) One can define the map $\eta$  by $\eta(\P)=\zeta(\P^c)$; in most cases $\zeta(\P)\neq \eta(\P)$.  Section~\ref{sec:zeta_map} is devoted to presenting the algorithms for calculating the zeta map and the eta map in multiple fashions.  In particular, we present two new methods involving lasers and interval intersections.

Meanwhile, $\zeta$ and $\eta$ combine to induce a new {\em area-preserving} involution $\chi$ on the set of Dyck paths defined in Section~\ref{sec:perp} by \[\chi(\Q):=\eta(\zeta^{-1}(\Q))=\zeta(\zeta^{-1}(\Q)^c).\]  In Section~\ref{sec:square}, we give a new proof that in the classical Catalan case, this {\em conjugate-area map} $\chi$ is the map that reverses the Dyck path.   Applying our inverse algorithm presents a new construction of the inverse of the zeta map on a Dyck path.  However, we have no explicit description of $\chi(\Q)$ from $\Q$ in the general $(a,b)$-case.  Indeed, a concrete construction of $\chi(\Q)$ from $\Q$ could be used to construct an explicit inverse for the zeta map.

In Section~\ref{sec:inductive_zeta_inverse}, we show that when a rational Dyck path $\Q$ visits the lattice point having level equal to 1, $\zeta^{-1}(Q)$ has a nice decomposition as does its image under the conjugate-area map $\chi$.  These observations allows us to explicitly find $\chi$ (and therefore $\zeta^{-1}$) of any path that has valleys exactly on levels equal to $\{1,\hdots,k\}$ for $k<a$ in Theorem~\ref{thm:inductive_area}.  We have also constructed $\chi(Q)$ and $\zeta^{-1}(Q)$ for paths that bound left-adjusted or up-adjusted partitions in Proposition~\ref{prop:justified}.

Section~\ref{sec:9} investigates the poset of rational Dyck paths ordered by when one path is weakly below the other, motivating a new statistic $\delta(P)$ that appears to be fruitful for recursively computing $\zeta^{-1}$ from evidence gathered by computer learning algorithms.  Indeed, in the remainder of Section~\ref{sec:9}, we use $\delta(P)$ to construct the initial part of a rational bounce path and to give a new algorithm that computes $\zeta^{-1}$ for $(a,am+1)$-Dyck paths. 

\medskip
One of the primary motivations for our research was the study of conjectured statistics for the $q,t$-enumeration of $(a,b)$-Dyck paths. Section~\ref{sec.notation} sets the stage by introducing key combinatorial concepts and statistics associated to $(a,b)$-Dyck paths.  In Section~\ref{sec:skew_length} we investigate the \emph{skew length} statistic $\sl(\P)$, originally defined in the context of $(a,b)$-cores in \cite{AHJ14}.  The original definition of skew length seems to depend on the ordering of $a$ and $b$; we show that skew length is in fact independent of this choice.  The main tools we develop involve a row length filling of the boxes under the $(a,b)$-Dyck path $\P$ and above the main diagonal, along with the idea of skew inversions and flip skew inversions.   Section~\ref{sec:conjugate} shows that \emph{skew length} is preserved under conjugation.  
 
\section{Background and Notation}
\label{sec.notation}

\begin{definition}
An {\em $(a,b)$-lattice path} $\P$ is a lattice path in $\ZZ^2$ consisting of north and east steps starting from the origin and ending at the point $(b,a)$.

We call $\P$ an {\em $(a,b)$-Dyck path} if $\P$ remains (weakly) above the diagonal line connecting the origin to $(b,a)$.  Equivalently, the lattice points $(x,y)$ along $\P$ satisfy $ax\leq by$. We draw $(a,b)$-Dyck paths in an $a\times b$ grid, where the lower left corner is the origin. 

We denote the full collection of $(a,b)$-Dyck paths by $\DD_{a,b}$, or simply $\DD$ if there is no confusion about the values of $a$ and $b$.
\end{definition}

We use the English notation for Young diagrams, drawing the largest row at the top.  The \emph{hook length} of a box $B$ in the Young diagram of a partition is the number of boxes in the \emph{hook} of boxes directly below or directly to the right of $B$, including the box $B$ itself.  An \emph{$a$-core partition} (or simply $a$-core) is a partition for which its Young diagram has no boxes with hook length equal to~$a$.  Similarly, a \emph{simultaneous $(a,b)$-core partition} (or $(a,b)$-core for short) has no hooks equal to $a$ or $b$.

Anderson proved that when $a$ and $b$ are relatively prime there are finitely many $(a,b)$-cores \cite{And02} by finding a bijection with the set of $(a,b)$-Dyck paths; these are counted by the formula
\[
\frac{1}{a+b}\binom{a+b}{a}.
\]
This formula seems to have been discovered at various times; the earliest reference we know of is~\cite{DM47} in 1947.  In 1954, Bizley considered the general case of rectangular Dyck paths of which this formula is a special case~\cite{Biz54}.

Bizley's counting method starts from the full set of lattice paths from $(0,0)$ to $(b,a)$, and considers the orbit of the cyclic group $C_{a+b}$ acting by \emph{cyclic shifts} on paths.  In the case where $a$ and $b$ are relatively prime, there is a unique Dyck path in each such orbit.

\begin{example} Let $N$ and $E$ represent a north step and an east step, respectively.  Throughout this paper, we will use as our running example the $(5,8)$-Dyck path \[\P=NNNENEEENEEEE,\] shown in Figure~\ref{fig:example58}.   
\end{example}

\begin{figure}
  \begin{center}
  \includegraphics[width=0.8\textwidth]{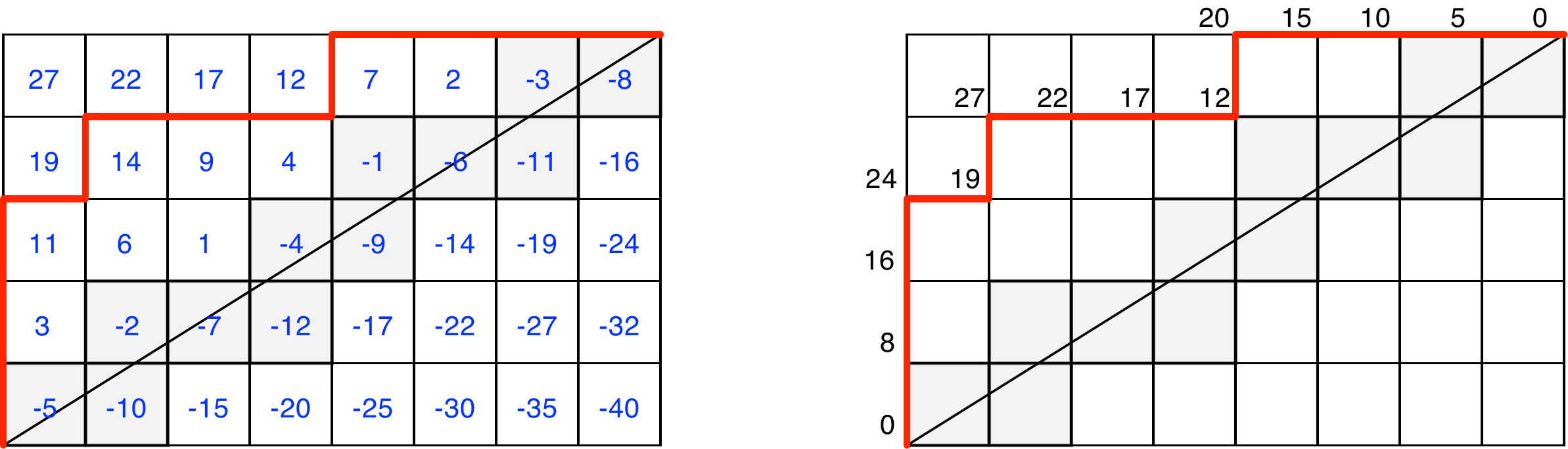} 
  \end{center}
  \caption{(Left) A lattice path $\P$ when $a=5$ and $b=8$.  The hook filling is given by the numbers in the center of the boxes.  The boxes above the path show that the partition bounded by $\P$ is $(4,1)$. \newline(Right) The levels of the lattice points along the path.}
  \label{fig:example58}
\end{figure}

\newpage\subsection{Dictionary of notation}\

We keep track of numerous bits of data associated to an $(a,b)$-Dyck path $\P$. 

\begin{enumerate}
\item {\bf General constructions:}

\begin{itemize}
\item The \emph{hook filling} of the boxes in the square lattice is obtained by filling the box with lower-right lattice point $(b,0)$ with the number $-ab$ and increasing by $a$ for every one box west and increasing by $b$ for every one box north. A box is above the main diagonal if and only if the corresponding hook is positive. (See Figure~\ref{fig:example58}.)

\begin{figure}
  \begin{center}
  \includegraphics[width=0.7\textwidth]{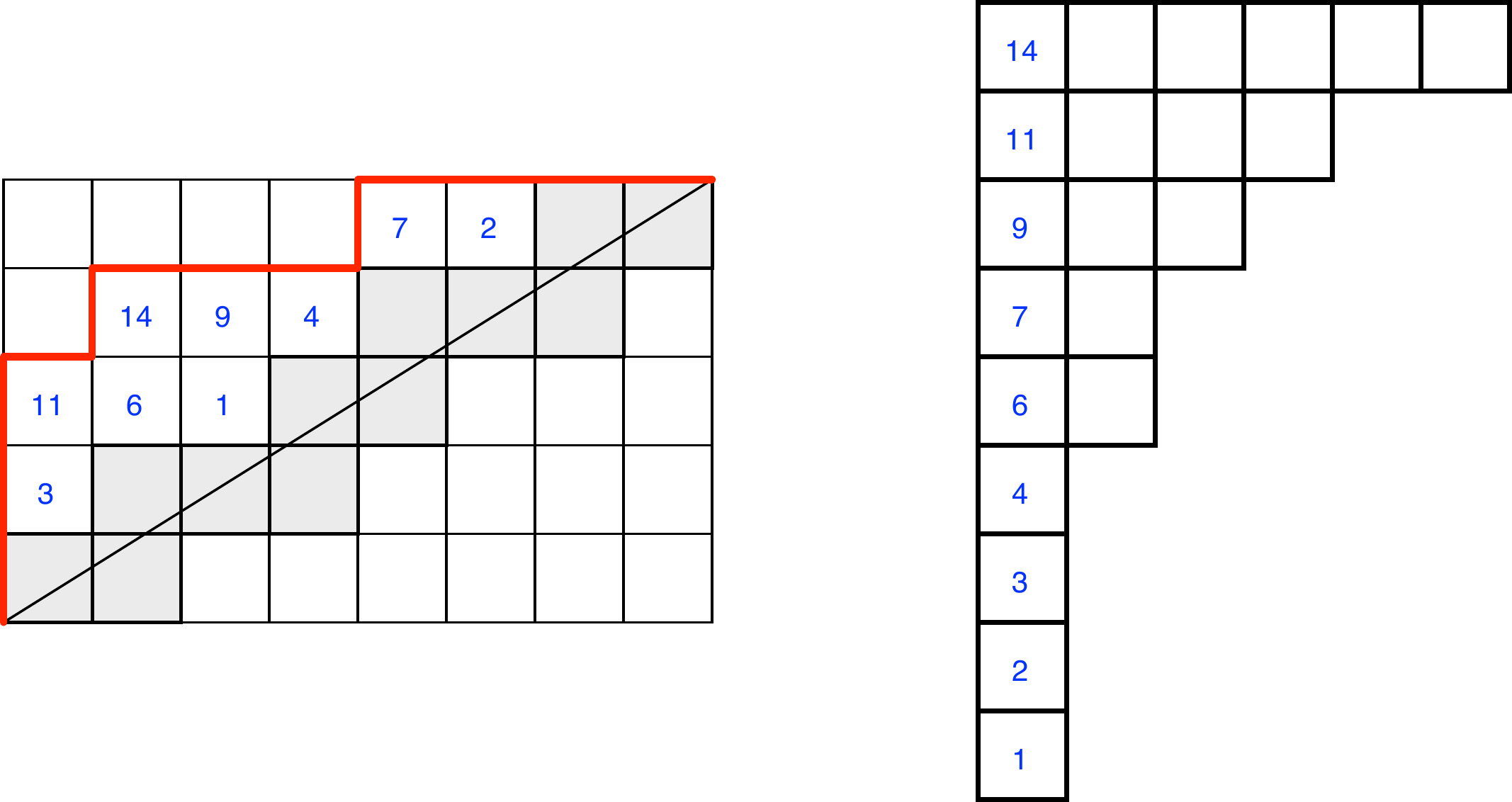}
  \end{center}
  \caption{Anderson's bijection gives a correspondence between $(a,b)$-Dyck paths and $(a,b)$-core partitions.  Corresponding to $\P=NNNENEEENEEEE$ is the $(5,8)$-core $(6,4,3,2,2,1,1,1,1)$.}
  \label{fig:AndersonsBijection}
\end{figure}

\item The \emph{positive hooks} of $\P$ are the numbers in the hook filling below the path but greater than zero.  (Elsewhere these have been called {\em beta numbers} or {\em bead numbers}.)  

\item We denote by $\core(\P)$ the $(a,b)$-core corresponding to $\P$ under Anderson's bijection. 
The hook lengths of the boxes in the first column of $\core(\P)$, its {\em leading hooks}, are precisely the positive hooks of $\P$. An example of Anderson's bijection is illustrated in Figure~\ref{fig:AndersonsBijection}.

\item The \emph{row length filling} of $\P$ are numbers placed in the boxes under $\P$. They correspond to the number of boxes in the row of $\core(\P)$ with the given hook. This will be developed in Section~\ref{sec:length_filling}.  (See Figure~\ref{fig:rowLength}.)
\item The \emph{partition bounded by} $\P$ is the partition whose Young diagram is the collection of boxes above the path $\P$.
\end{itemize}

\item {\bf Combinatorial statistics:}
\begin{itemize}
\item The \emph{area} of $\P$, denoted $\area(\P)$ is the number of positive hooks of $\P$. Equivalently, this is the number of rows in $\core(\P)$.

\item The \emph{rank} of $\P$, denoted $\rk(\P)$ is the number of rows in the partition bounded by $\P$.
\item The \emph{skew length} of $\P$, denoted $\sl(\P)$ is a statistic that we discuss in detail in Section~\ref{sec:skew_length}.
\end{itemize}
\item {\bf Sets and sequences of numbers associated to $\P$:}
\begin{itemize}

\item The \emph{levels} of $\P$ are labels associated to the lattice points of $\P$ defined by Armstrong, Loehr, and Warrington in \cite{ALW-parking,ALW-sweep}.  Assign level $0$ to $(0,0)$ and label the other lattice points of $\P$ by adding $b$ after each north step and subtracting $a$ after each east step.  Equivalently, this is the value of the hook filling in the box to the northwest of the lattice point.  Note that the label of the northeast-most lattice point $(b,a)$ is once again $0+a\cdot b-a \cdot b=0$.

\item The path $\P$ has two reading words obtained by reading the levels in order.  The \emph{reading word} of $\P$, denoted $L(\P)$ (for `levels'), is obtained by reading the levels that occur along the path from southwest to northeast, excluding the final $0$.  (One can imagine assigning to each north and east step of the path the level of the step's initial lattice point.)

The \emph{reverse reading word}, denoted $M(\P)$, is obtained by reading from northeast to southwest, excluding the final $0$. (One can imagine $\P$ as a path from $(b,a)$ to $(0,0)$ consisting of west and south steps, once again assigning to each step  the level of its initial lattice point.)

Reading along $\P$ in Figure~\ref{fig:example58} shows that
\[L(\P)=(0,8,16,{24},19,{27},{22},{17},12,{20},{15},{10},{5})\] and 
\[M(\P)=({0},{5},{10},{15},20,{12},{17},{22},27,{19},24,16,8).\]
When $a$ and $b$ are relatively prime, no value occurs more than once in $L(\P)$ or $M(\P)$.

\item The set of levels of $\P$ is partitioned into the set of \emph{north levels} $\N(\P)$ and \emph{east levels} $\E(\P)$, where when reading from southwest to northeast, levels of lattice points starting north steps of $\P$ are in $\N(\P)$ and levels of lattice points starting east steps of $\P$ are in $\E(\P)$. We order these levels in decreasing order.  In our running example, the north levels of $\P$ are 
\[\N(\P)=\{19,16,12,8,0\},\] 
and the east levels of $\P$ are 
\[\E(\P)=\{27,24,22,20,17,15,10,5\}.\]
\end{itemize}

\item {\bf Permutations associated to $\P$:}
Throughout the paper we use square brackets to write permutations in one-line notation, and round parentheses for permutations in cycle notation.\smallskip

\begin{itemize}

\item The \emph{reading permutation} of $\P$ is a permutation $\sigma$ in $S_{a+b}$ that encodes the relative order of the levels recorded in $L(\P)$. The \emph{reverse reading permutation} of $\P$, denoted~$\tau(\P)$, encodes the relative order of the values in~$M(\P)$. In our running example, the one-line notation for $\sigma(\P)$ and $\tau(\P)$ are 
\[\sigma(\P)=[1,3,7,{12},9,{13},{11},{8},5,{10},{6},{4},{2}]\] 
and 
\[\tau(\P)=[{1},{2},{4},{6},10,{5},{8},{11},13,{9},12,7,3].\] 

\item Let $\gamma(\P)$ be the permutation in $S_{a+b}$ that when written in cycle notation starting with $1$ has the same order of entries as $\sigma(\P)$ written in one-line notation.  In our running example $\P$ we have
\begin{align*} 
\gamma(\P)&= (1,3,7,{12},9,{13},{11},{8},5,{10},{6},{4},{2})\\ 
 &=  [3,1,7,2,10,4,12,5,13,6,8,9,11].
\end{align*}
\end{itemize}
\end{enumerate}

\renewcommand*{\thefootnote}{\fnsymbol{footnote}}
\setcounter{footnote}{1}

\begin{remark}\label{rem:sigma}
The path $\P$ can be recovered knowing only $\sigma(\P)$ (or $\tau(\P)$ or $\gamma(\P)$).  
The east steps of~$\P$ correspond exactly to the right (cyclic) descents\footnote{A descent of a permutation occurs when $\sigma(i)>\sigma(i+1)$.  A cyclic descent is defined in the same way, but considering the indices modulo $a+b$, allowing a descent in the last position of $\sigma$.} of $\sigma$; whereas, the north steps of $\P$ correspond to the right (cyclic) ascents of $\sigma$. In our running example, the right (cyclic) descents of $\sigma(\P)$ occur in positions $4$, $6$, $7$, $8$, $10$, $11$, $12$, and $13$, which are exactly the positions of the east steps in $\P$.
\end{remark}

\section{Skew length}\label{sec:skew_length}

In~\cite{AHJ14} the \emph{skew length} statistic is proposed as a $q$-statistic for $(a,b)$-Dyck paths and a related construction is investigated in \cite[Section 4]{ALW-parking}. In this section, we present the original definition of skew length on cores and two equivalent interpretations on $(a,b)$-Dyck paths using length fillings and skew inversions.  We show that these interpretations are indeed equivalent to the original definition and, as a consequence, we prove that skew length is independent of the ordering of $a$ and $b$.  Further interpretations of skew length are presented in terms of the zeta map in Section~\ref{sec:zeta_map}.

\subsection{Skew length on cores and polynomial motivation}\

We begin with an observation on ordinary core partitions before discussing simultaneous core partitions.

\begin{definition}[{\cite[Definition 2.7]{AHJ14}}]
Let $\kappa$ be an $a$-core partition.  Consider the hook lengths of the boxes in the first column of $\kappa$.  Find the largest hook length of each residue modulo $a$.  The \emph{$a$-rows} of $\kappa$ are the rows of $\kappa$ corresponding to these hook lengths. The \emph{$a$-boundary} of $\kappa$ consists of all boxes in its Young diagram with hook length less than $a$.  \end{definition}

\begin{proposition}
\label{prop:arows}
Let $\kappa$ be an $a$-core partition.  The number of boxes in the $a$-rows of $\kappa$ equals the number of boxes in the $a$-boundary of $\kappa$.
\end{proposition}

\begin{proof}
Let $\len(h)$ be the number of boxes in the row of $\kappa$ with leading hook $h$.

We first observe that if $h>a$ is a leading hook of $\kappa$, then $h-a$ is also a leading hook of $\kappa$. For this, decompose $h$ into two hooks of lengths $h-a$ and $a$ as illustrated in Figure~\ref{fig:skewLength_proof}, such that the boxes in the row with leading hook $h$ that are intersected by the hook $a$ are exactly the boxes in the $a$-boundary in that row. This guarantees that the right-end box of the hook $h-a$ is in~$\kappa$, and therefore that $h-a$ is also a leading hook.

Now, the number of $a$-boundary boxes in the row of $\kappa$ corresponding to $h$ is $\len(h)-\len(h-a)$. Summing over all rows gives the number of $a$-boundary boxes; telescoping over residues modulo $a$ gives the number of boxes in the $a$-rows of $\kappa$.
\end{proof}

\begin{figure}
  \begin{center}
  \includegraphics[scale=0.5]{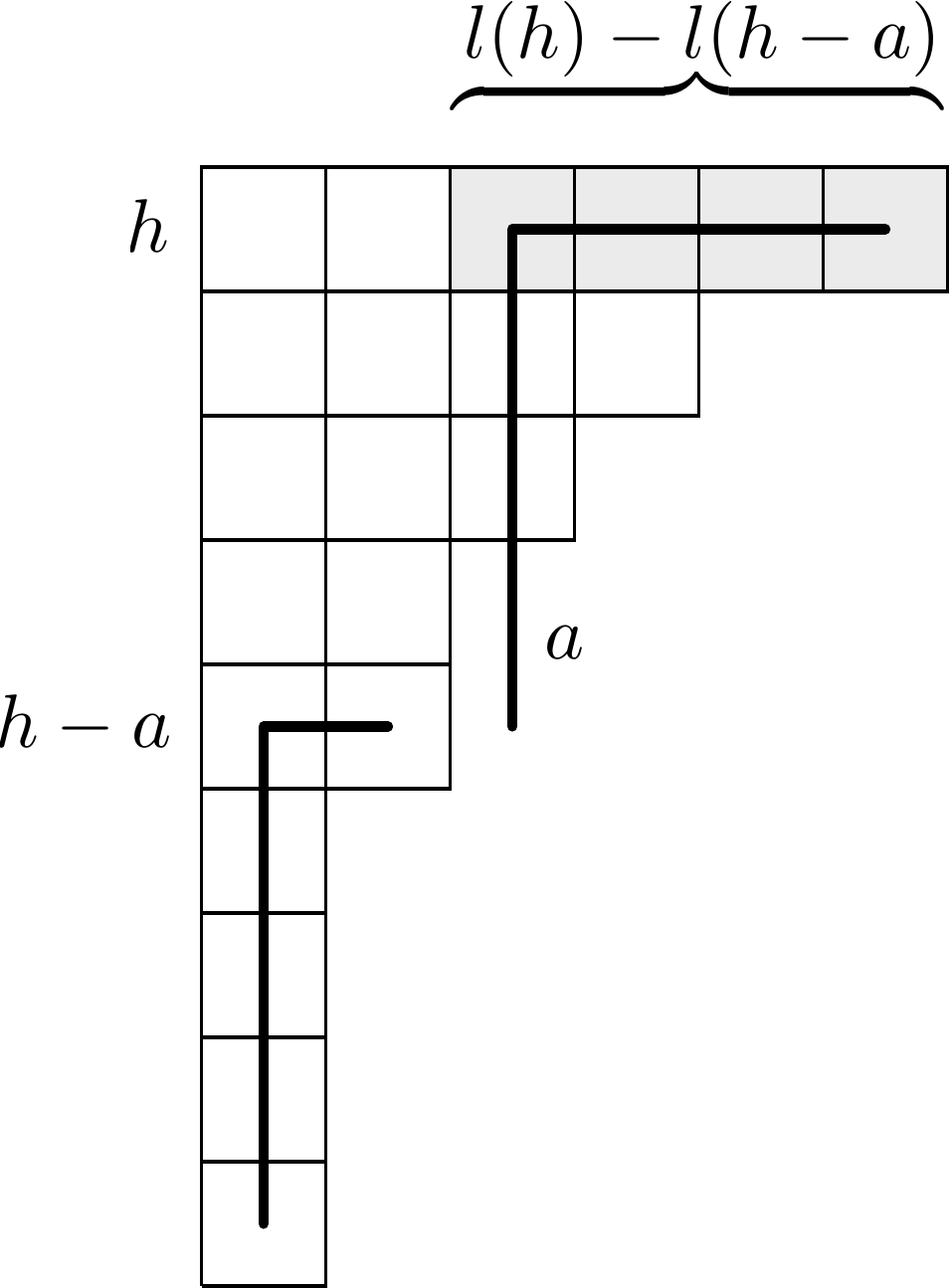}
  \end{center}
  \caption{The number of $a$-boundary boxes in the row of $\kappa$ corresponding to a leading hook $h$ is $\len(h)-\len(h-a)$.}
  \label{fig:skewLength_proof}
\end{figure}

\begin{corollary}
The number of boxes in the $a$-rows of $\kappa$ equals the number of boxes in the $a$-rows of $\kappa^c$
\end{corollary}

\begin{remark}
For readers familiar with the abacus diagram interpretation, hook lengths correspond to beads on the abacus; the $a$-rows correspond to the largest bead on each runner of the $a$-abacus.  Proposition~\ref{prop:arows} gives a way to count the number of boxes in the $a$-boundary of an $a$-core by adding the number of gaps that appear on the abacus before each of these largest beads.
\end{remark}

\begin{definition}[{\cite[Definition 2.7]{AHJ14}}]
\label{def:skewlen}
Let $\kappa$ be an $(a,b)$-core partition. The \emph{skew length} of~$\kappa$, denoted $\sl(\kappa)$, is the number of boxes simultaneously located in the $a$-rows and the $b$-boundary of~$\kappa$.
\end{definition}

\begin{example}
The core partition shown in Figure~\ref{fig:skewLength_cores} is the $(5,8)$-core $\kappa=\core(\P)$ corresponding to the path $\P$ in our running example from Figures~\ref{fig:example58} and \ref{fig:AndersonsBijection}. 

\begin{figure}
  \begin{center}
  \includegraphics[scale=0.6]{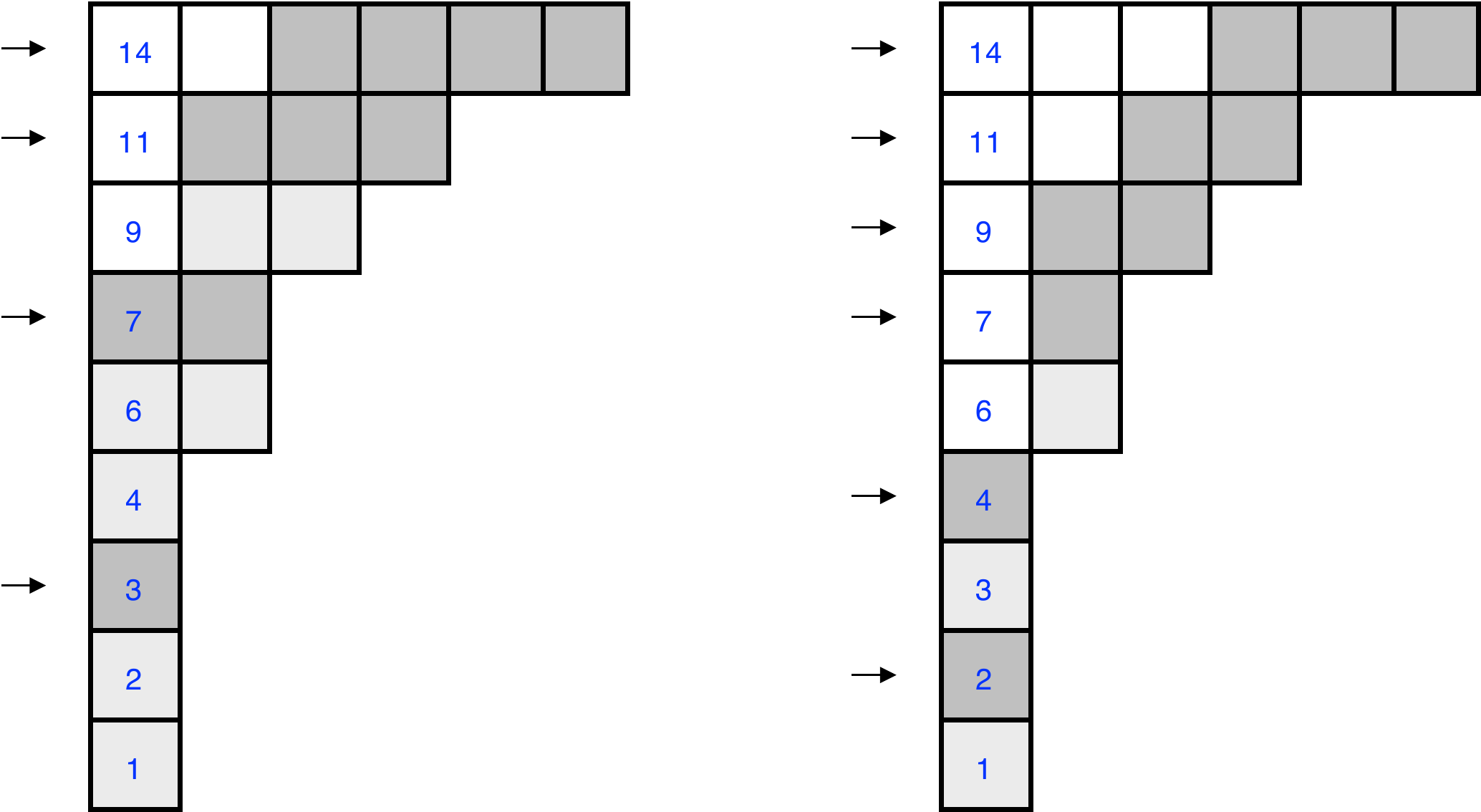}
  \end{center}
  \caption{(Left) The $8$-boundary boxes of our favorite $(5,8)$-core $\kappa$ are shaded; those in the $5$-rows of $\kappa$ are darker. 
  (Right) The $5$-boundary boxes of $\kappa$ are shaded; those in the $8$-rows of $\kappa$ are darker.
  Surprisingly, the number of darkly shaded boxes on the left $4+3+2+1=10$ is equal to the number of darkly shaded boxes on the right $3+2+2+1+1+1=10$. (See Corollary~\ref{cor.absl}.)}
  \label{fig:skewLength_cores}
\end{figure}

On the left, the $5$-rows of $\kappa$ are the rows with leading hook lengths 14, 11, 7, and 3.   The darkly shaded boxes are those boxes in the $5$-rows with hook length less than 8.  The skew length is equal to $4+3+2+1=10$.  

On the right, we compute of the skew length of $\kappa$ when considered as an $(8,5)$-core. The $8$-rows of $\kappa$ are the rows with leading hook lengths 14, 11, 9, 7, 4, and 2. The shaded boxes are those boxes in the $8$-rows with hook length less than 5. The skew length is equal to $3+2+2+1+1+1=10$.

We will see in Corollary~\ref{cor.absl} that it is not a coincidence that these two numbers are the same.  

The number of boxes in the $8$-boundary (shaded boxes, left) equals the number of boxes in the $8$-rows (marked rows, right) and the number of boxes in the $5$-boundary (shaded boxes, right) equals the number of boxes in the $5$-rows (marked rows, left), as proved in general in Proposition~\ref{prop:arows}.
\end{example}

The skew length statistic was found by Armstrong; he conjectures it as a key statistic involved in the $q$- and $q,t$-enumeration of $(a,b)$-cores (or $(a,b)$-Dyck paths). Recall that the rank $\rk(\kappa)$ of an $(a,b)$-core $\kappa$ is the number of rows in its corresponding Young diagram. 

\begin{conjecture}\cite[Conjecture 2.8]{AHJ14}
\label{conj.qcount}
Let $a$ and $b$ relatively prime positive integers. 
The expression
\[
f_{a,b}(q) = \frac{1}{[a+b]_q}\begin{bmatrix}a+b\\ a\end{bmatrix}_q
\]
is equal to the polynomial \[g_{a,b}(q) = \sum_\kappa q^{\sl(\kappa)+\rk(\kappa)},\] where the sum is over all $(a,b)$-cores $\kappa$.
\end{conjecture}

Haiman \cite[Propositions 2.5.2 and 2.5.3]{Haiman93} proved that $f_{a,b}(q)$ is a polynomial if and only if $a$ and $b$ are relatively prime. 
\cite[Theorem 1.10]{BEG03} provides a proof that $f_{a,b}(q)$ has non-negative coefficients involving representation theory of rational Cherednik algebras, see also~\cite[Section 1.12]{gordon_catalan_2012}. A proof of Conjecture~\ref{conj.qcount} would provide a combinatorial interpretation for the coefficients of $f_{a,b}(q)$.

\begin{proposition}\cite{Haiman93,BEG03}
\label{prop.qdef}
 The expression
\[
f_{a,b}(q) = \frac{1}{[a+b]_q}\begin{bmatrix}a+b\\ a\end{bmatrix}_q
\]
is a polynomial if and only if $\gcd(a,b)=1$.  Furthermore, when $a$ and $b$ are relatively prime, the resulting polynomial has integer coefficients.
\end{proposition}

Define the \emph{co-skew length} of an $(a,b)$-core $\kappa$ as 
\[
\sl'(\kappa):=\frac{(a-1)(b-1)}{2}-\sl(\kappa).
\]
Armstrong conjectures that rank and co-skew length give a $q,t$-enumeration of the $(a,b)$-cores, subject to the following symmetry:
\begin{conjecture}\cite[Conjecture 2.9]{AHJ14}
\label{conj.qtsym}
The following $q,t$-polynomials are equal:
\[
\sum q^{\rk(\kappa)}t^{\sl'(\kappa)} = \sum q^{\sl'(\kappa)}t^{\rk(\kappa)}
\]
where the sum is over all $(a,b)$-cores $\kappa$.
\end{conjecture}
These $q,t$-polynomials are called the \emph{rational $q,t$-Catalan numbers}.

\subsection{Skew length on Dyck paths via the row length filling}\label{sec:length_filling}\

We now provide a new method to calculate the skew length of an $(a,b)$-Dyck path $\P$ which uses a {\em row length filling} of the boxes below~$\P$. Our method recovers with the skew length statistic discovered by Armstrong for $(a,b)$-cores. As a consequence, we show that skew length of an $(a,b)$-core is independent of the ordering of $a$ and $b$.

We provide two equivalent definitions of the {\em row length filling}.  

\begin{definition}\label{def:rlf}
Let $\P$ be an $(a,b)$-Dyck path.  The \emph{row length filling} of $\P$ is an assignment of numbers to each box below the path~$\P$.  

For a box $B$ with positive hook filling $h$, define the row length of $B$ to be the length of the row in $\core(\P)$ with leading hook $h$.  Alternatively, define the row length of $B$ to be $h-p_h$, where $p_h$ is the number of positive entries in the hook  filling strictly less than~$h$.  

For a box $B$ with non-positive hook filling $h$, define the row length of $B$ to be zero.

For any hook $h$ in the hook filling of $\P$, we denote by~$\len(h)$ the corresponding value of the row length filling of $\P$.
\end{definition}

Figure~\ref{fig:rowLength} shows in red in the upper left corner the row length of the boxes corresponding to the positive hooks of $\P$.
\begin{figure}
  \begin{center}
  \includegraphics[scale=0.6]{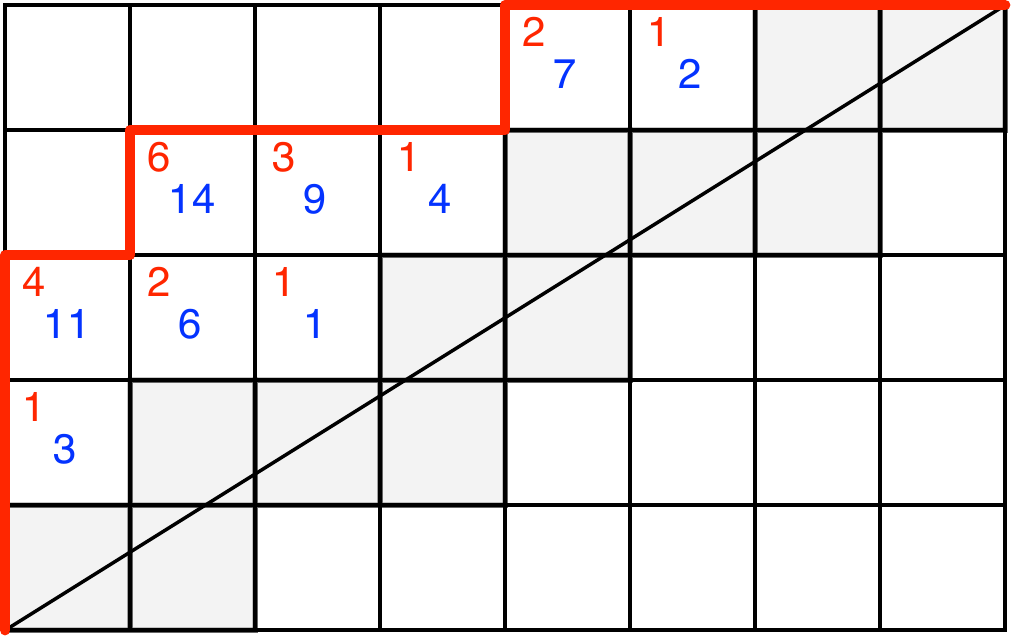}
  \end{center}
  \caption{The row length filling of boxes below the path $\P$ is given in red in the upper left corner.  The values correspond to the length of the rows of $\core(\P)$ in Figure~\ref{fig:AndersonsBijection}.}
  \label{fig:rowLength}
\end{figure}

\begin{lemma}
The two definitions of row length filling in Definition~\ref{def:rlf} are equivalent.
\end{lemma}

\begin{proof}
When ordered in increasing order, the entries in the hook filling of $\P$ correspond to the hook lengths of the boxes in the first column of $\core(\P)$ from shortest to longest. Suppose the first box of the $i$th shortest row has hook length $h$. Then the length of the $i$th shortest row is $h-(i-1)$, which is exactly the corresponding entry in the row length filling.
\end{proof}

\begin{remark}
For readers familiar with the abacus diagram interpretation, the row length filling associates to each bead on the abacus the number of gaps that appear before it on the abacus.
\end{remark}

The row length filling is very useful for reading off common core statistics 
from the Dyck path.  For example, we can immediately see that:
\begin{corollary}
The sum of the entries of the row length filling of $\P$ is equal to the number of boxes of the core $\core(\P)$.
\end{corollary}
Furthermore, because the $a$-rows of $\core(\P)$ correspond to the westmost boxes under $\P$ and the $b$-rows of $\core(\P)$ correspond to the northmost boxes under $\P$, the number of boxes in $\core(\P)$ with hook length less than $a$ or less than $b$ can be determined from the row length filling as a direct consequence of Proposition~\ref{prop:arows}.
\begin{corollary}
\label{cor.boundaryBoxes}
The number of boxes in the $a$-boundary of an $(a,b)$-core $\core(\P)$ is equal to the sum of the row length fillings of the westmost boxes under $\P$.  Likewise, the number of boxes in the $b$-boundary of $\core(\P)$ is equal to the sum of the row length fillings of the northmost boxes under $\P$.
\end{corollary}

In the same vein, the skew length of $\P$ can also be easily computed, as follows: 

\begin{theorem}
\label{thm.skewLengthCompute}
The skew length of an $(a,b)$-core $\core(\P)$ may be computed from the row length filling of~$\P$ by adding all lengths at peaks of $\P$ and subtracting all lengths at valleys of $\P$.
\end{theorem}

\begin{proof}
By the argument in the proof of Proposition~\ref{prop:arows}, we see that when $h$ is a positive hook of an $(a,b)$-Dyck path $\P$ (so that $h-a$ is the hook of the box directly east of the box with hook $h$ and $h-b$ is the hook of the box directly south of the box with hook $h$), then
\begin{enumerate}[(i)]
\item The number of $a$-boundary boxes in the row of $\core(\P)$ corresponding to $h$ is $\len(h)-\len(h-a)$.
\item The number of $b$-boundary boxes in the row of $\core(\P)$ corresponding to $h$ is $\len(h)-\len(h-b)$.
\end{enumerate}
By restricting to the $a$-rows or $b$-rows, we see that the skew length of $\core(\P)$ is given by:
\begin{equation}\label{eqn.lenSum_a}
\sum \len(h)-\len(h-b),
\end{equation}
where the sum is over all westmost boxes under $\P$, or alternatively the skew length of $\core(\P)$ is given by:
\begin{equation}\label{eqn.lenSum_b}
\sum \len(h)-\len(h-a),
\end{equation}
where the sum is over all northmost boxes under $\P$.
When one westmost box under $\P$ is directly north of another, Formula~\eqref{eqn.lenSum_a} telescopes.  After cancelling terms, we are left with the lengths at peaks of $\P$ minus the lengths at valleys of $\P$.  An equivalent argument can be made from Formula~\eqref{eqn.lenSum_b}.
\end{proof}

\begin{example}
In Figure~\ref{fig:rowLength}, we see that the sum of the row length fillings is 21, which is the number of boxes of $\core(\P)$.  Adding the row lengths of the westmost boxes under $\P$ gives $2+6+4+1+0=13$ boxes in the $5$-boundary of $\core(\P)$, while adding the row lengths of the northmost boxes under $\P$ gives $4+6+3+1+2+1+0+0=17$ boxes in the $8$-boundary of $\core(\P)$, as expected from Figure~\ref{fig:skewLength_cores}. Our path $\P$ has three peaks with row lengths 2, 6, and 4 and two valleys with row lengths 2 and 0. 
The skew length of our path is then
\[\sl(\P)=(2+6+4)-(2+0)=10.\]  
\end{example}

When computing skew length directly from the core, it is not obvious that the number of boxes in $a$-rows and the $b$-boundary should be equal to the number of boxes in $b$-rows and the $a$-boundary (see Figure~\ref{fig:skewLength_cores}).  But the method of computing the skew length given by Theorem~\ref{thm.skewLengthCompute} is independent of the ordering of $a$ and $b$: Switching $a$ and $b$ flips the rectangle to a $b\times a$ rectangle in which peaks are still peaks, valleys are still valleys, and the hook filling and row length filling are otherwise unaffected.

\begin{corollary}
\label{cor.absl}
The skew length of an $(a,b)$-core $\kappa$ is independent of the ordering of $a$ and $b$.
\end{corollary}

\subsection{Skew length via skew inversions}\

This section presents another interpretation of the skew length of an $(a,b)$-Dyck path $\P$ in terms of the number of its {\em skew inversions} or the number of its {\em flip skew inversions}. 

Recall that the north levels of $\P$ are the levels $\N(\P)=\{n_1,\dots,n_a\}$ of the initial lattice points of the north steps in the path, and that the east levels of $\P$ are the levels $\E(\P)=\{e_1,\dots,e_b\}$ of the initial lattice points of the east steps.

\begin{definition}
A \emph{skew inversion} of $\P$ is a pair of indices $(i,j)$ such that $n_i>e_j$.  
A \emph{flip skew inversion} of $\P$ is a pair of indices $(i,j)$ with $n_i+b<e_j-a$. 
\end{definition}

\begin{theorem}\label{thm:skewLength-skewInversions} Let $\P$ be an $(a,b)$-Dyck path. The skew length of $\P$ equals the number of skew inversions of $\P$, which is equal to the number of flip skew inversions of $\P$.
\end{theorem}

The key to the proof of Theorem~\ref{thm:skewLength-skewInversions} is recognizing the relationship between westmost boxes under $\P$ and north levels in $\N(\P)$ and the relationship between northmost boxes under $\P$ and east levels in $\E(\P)$. 

\begin{remark}
\label{rem:hook-north-east}
Figure~\ref{fig:hook-north-east} shows that when $h$ is the hook filling of a westmost box under $\P$, then the associated north level $n_h$ (corresponding to the lattice point at its southwest corner) is $h+a$.  When $h$ is the hook filling of a northmost box under $\P$, then the associated east level $e_h$ (corresponding to the lattice point at its northwest corner) is $h+a+b$. 
\begin{figure}
  \begin{center}
  \includegraphics[scale=0.5]{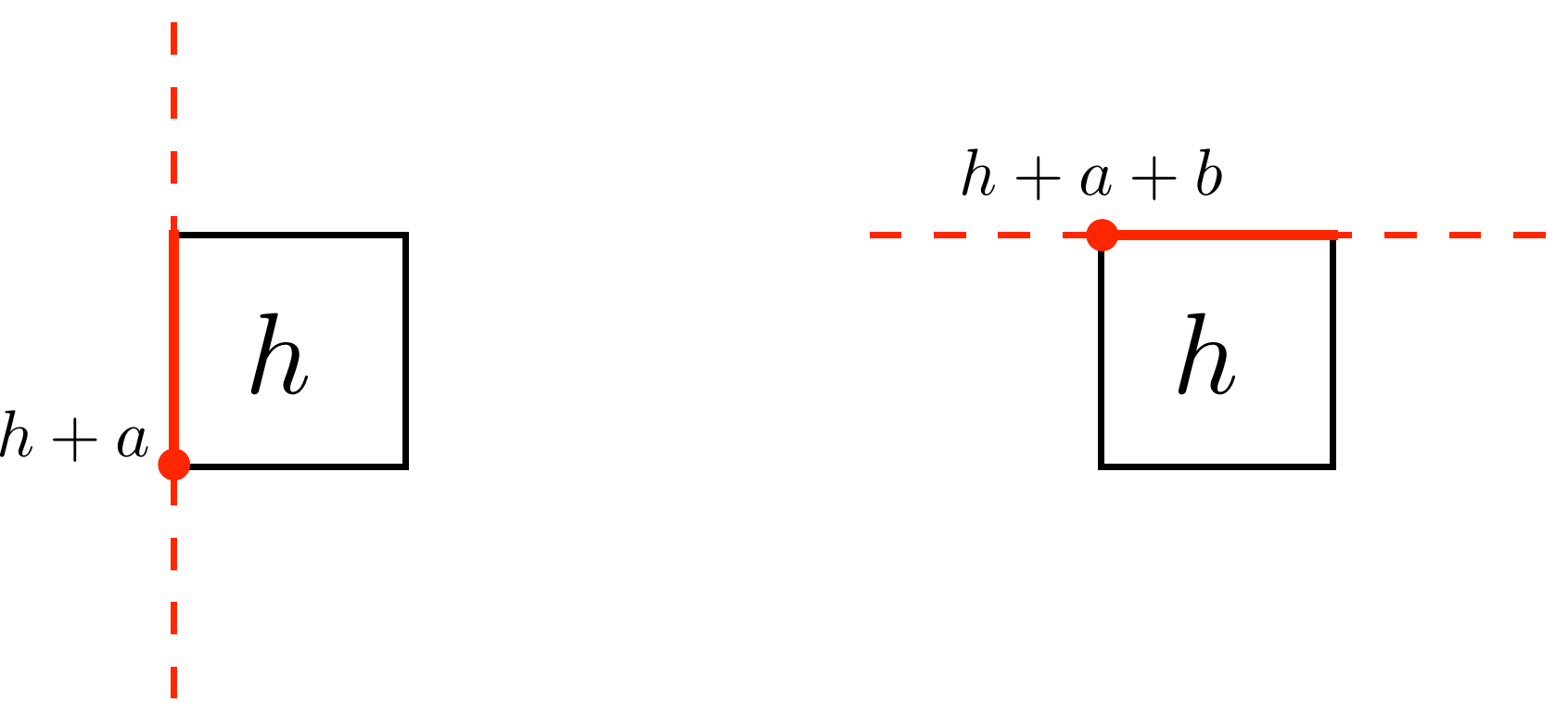}
  \end{center}
  \caption{When $h$ is the hook filling of a westmost box under $\P$, the associated north level is $n_h=h+a$.  When $h$ is the hook filling of a northmost box under $\P$, the associated east level is $e_h=h+a+b$.}
  \label{fig:hook-north-east}
\end{figure}
\end{remark}

\begin{lemma}\label{lem:skew_inversions_length_dif}
Let $h$ be the hook filling of a westmost box under an $(a,b)$-Dyck path $\P$. The length difference $\len(h)-\len(h-b)$ is equal to the number of skew inversions involving the associated north level $n_h$, which equals the number of $b$-boundary boxes in the $a$-row corresponding to $h$.
\end{lemma}

\begin{proof}
Recall that $\len(h)=h-p_h$, where $p_h$ is the number of positive hooks in the hook filling of $\P$ less than $h$.  Then:
\begin{eqnarray*}
\len(h)-\len(h-b)&=&h-p_h - (h-b) + p_{h-b}\\
  &=&b-(p_h-p_{h-b})\\
  &=&b-\#\{g\mid h-b\leq g< h\}.
\end{eqnarray*}

Each box with hook filling $g$ satisfying the inequalities $h-b\leq g< h$ is in a distinct column of the diagram of $\P$. If two were in the same column, then the difference of their hooks would be a multiple of $b$, so that both could not satisfy the inequality.  As a result, we may add a multiple of $b$ to each $g$ satisfying the inequalities to obtain a unique northmost box  under $\P$ with hook filling $\bar{g}$ satisfying $h-b\leq \bar{g}$. Conversely, for every northmost box under $\P$ with hook filling $\bar{g}$ satisfying this inequality there is a unique box with hook filling $g$ in the same column satisfying $h-b\leq g< h$. Therefore,

\begin{equation*}
\begin{aligned}
\len(h)-\len(h-b)&=b-\#\{\bar{g}\mid h-b\leq \bar{g}\}\\
&=\#\{\bar{g}\mid h-b> \bar{g}\}.
\end{aligned}
\end{equation*}

By Remark~\ref{rem:hook-north-east}, this is equivalent to $\len(h)-\len(h-b)=\#\{e_j\mid n_h> e_j\}$, as desired.  The last clause of the statement of the lemma is given in the proof of Theorem~\ref{thm.skewLengthCompute}.
\end{proof}

Similar arguments prove the following.
\begin{lemma}\label{lem:skew_inversions_length_dif_b}
Let $h$ be the hook filling of a northmost box under an $(a,b)$-Dyck path $\P$. The length difference $\len(h)-\len(h-a)$ is equal to the number of flip skew inversions involving the associated east level $e_h$, which equals the number of $a$-boundary boxes in the $b$-row corresponding to $h$.
\end{lemma}

Theorem~\ref{thm:skewLength-skewInversions} now follows directly from Definition~\ref{def:skewlen} by summing over all westmost boxes in Lemma~\ref{lem:skew_inversions_length_dif} and all northmost boxes in Lemma~\ref{lem:skew_inversions_length_dif_b}.

\begin{example}\label{ex:skew_inversions}
In our running example, the north levels are~$\N=\{19,16,12,8,0\}$ and the east levels are~$\E=\{27,24,22,20,17,15,10,5\}$. There are 10 skew inversions because there are 4 east levels less than $n_1=19$, 3 east levels less than~$n_2=16$, 2 east levels less than~$n_3=12$, 1 east level less than~$n_4=8$, and 0 east levels less than~$n_5=0$. The total number of skew inversions is then $4+3+2+1+0=10$. 
These numbers correspond to the number of $b$-boundary boxes in the $a$-rows of the core $\core(\P)$ in Figure~\ref{fig:skewLength_cores}.

To calculate the flip skew inversions, consider the sets $\N+b=\{27,24,20,16,8\}$ and $\E-a=\{22,19,17,15,12,10,5,0\}$.  There are 10 flip skew inversions because there are 3 elements of the form~$n_i+b$ less than~$e_1-a=22$, there are 2 less than~$e_2-a=19$, 2 less than~$e_3-a=17$, 1 less than~$e_4-a=15$, 1 less than~$e_5-a=12$, 1 less than~$e_6-a=10$, 0 less than~$e_7-a=5$, and 0 less than~$e_8-a=0$. The total number of flip skew inversions is then $3+2+2+1+1+1+0+0=10$. 
These numbers correspond to the number of $a$-boundary boxes in the $b$-rows of the of the core $\core(\P)$.
\end{example}

\begin{remark} Skew inversions in an $(a,b)$-Dyck path arise from pairs of north levels and east levels where $n_i>e_j$.  Note that $n_i+b$ is the level of the terminal lattice point of the corresponding north step (instead of initial lattice point), while $e_j-a$ is the level of the terminal lattice point of the corresponding east step.  So flip skew inversions are best understood by a reverse reading of $\P$ as a sequence of west and south steps, counting the pairs where the south level is less than the west level.  Alternatively, flip skew inversions of $\P$ correspond to skew inversions of $\P$ when $\P$ is reflected (flipped) to be a $(b,a)$-Dyck path. 
\end{remark}

\section{The conjugate map}
\label{sec:conjugate}

For any partition $\kappa$, its conjugate partition $\kappa^c$ is obtained by reflecting along its main diagonal. (See Figure~\ref{fig:conjugate_cores}.)  Since hook lengths are preserved under this reflection, when $\kappa$ is an $(a,b)$-core, so is $\kappa^c$.  When $a$ and $b$ are relatively prime, there is a natural conjugate map on $(a,b)$-Dyck paths~$\P$.  Apply cyclic shifts to the path $\P$ until we encounter a path strictly \emph{below} the diagonal, the conjugate path $\P^c$ is the result of rotating this path~$180^\circ$. (See Figure~\ref{fig:conjugate_paths}.) The first main result of this section (Theorem~\ref{thm:conjugation}) shows that these conjugations are equivalent under Anderson's bijection, and the second (Theorem~\ref{thm.slconj}) shows that conjugation preserves skew length. 
These two results were simultaneously found in independent work by Xin in~\cite{xin_rank_2015}.  Lemmas \ref{lem:conjugate_hooks} and \ref{lem:conjugate_positive_hooks} mirror the notion of conjugation of the semimodule of leading hooks presented by Gorsky and Mazin \cite{GMII}.

\begin{figure}
  \begin{center}
  \includegraphics[scale=0.5]{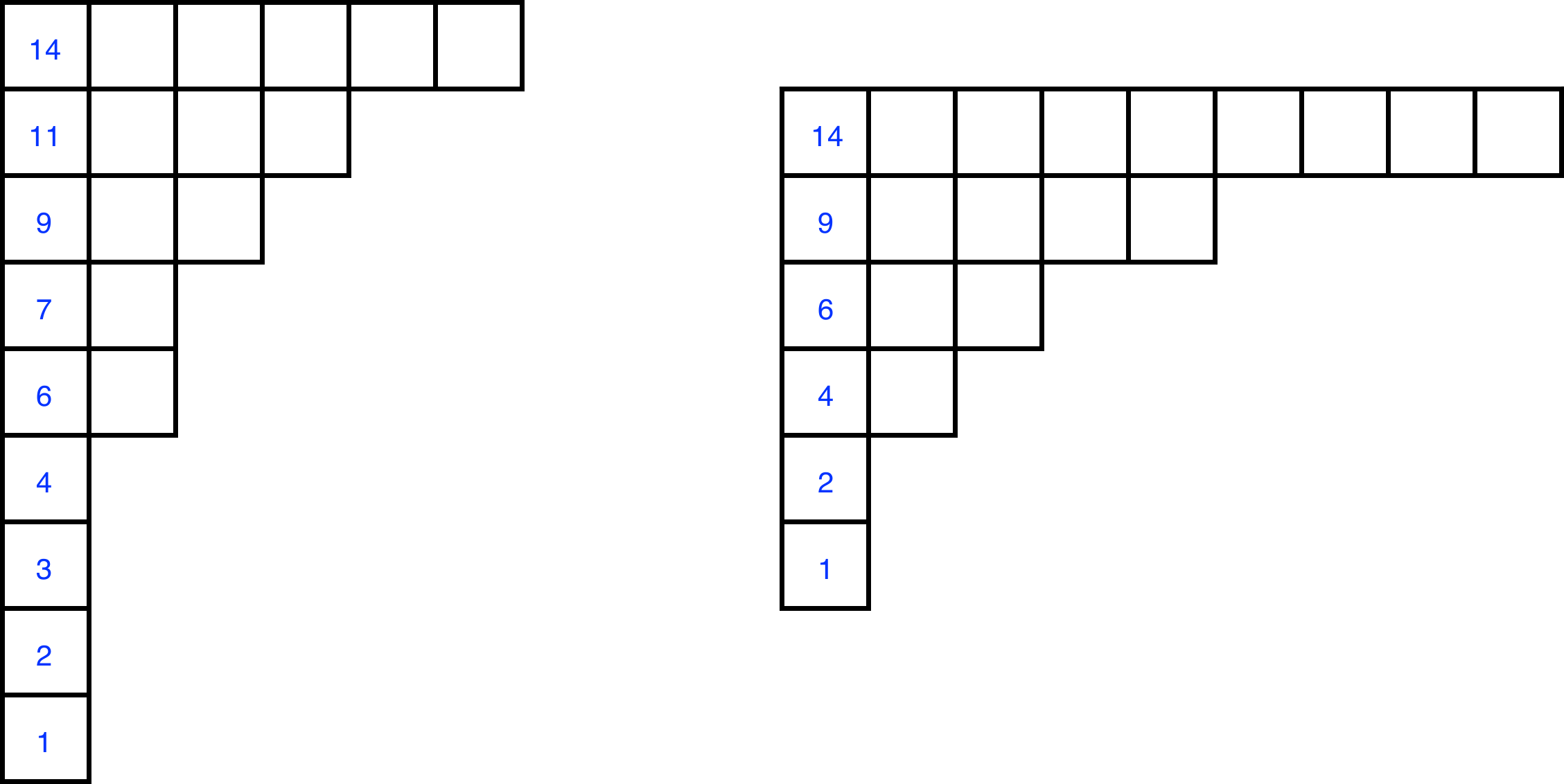}
  \end{center}
  \caption{The conjugate map on $(a,b)$-cores.}
  \label{fig:conjugate_cores}
\end{figure}

\begin{figure}
  \begin{center}
  \includegraphics[scale=0.5]{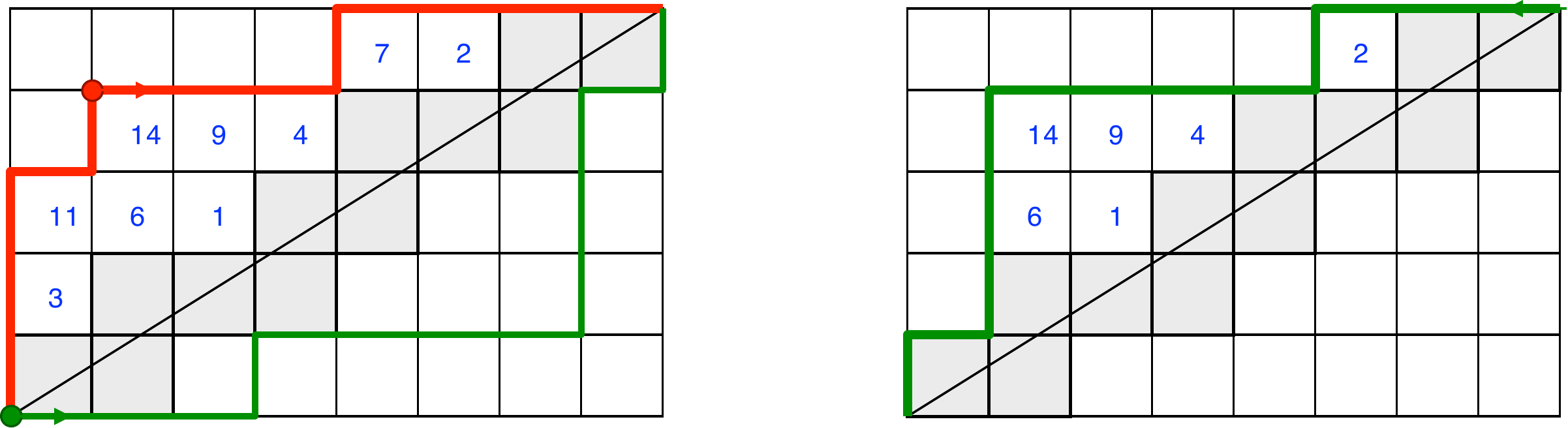}
  \end{center}
  \caption{The conjugate map on $(a,b)$-Dyck paths.}
  \label{fig:conjugate_paths}
\end{figure}

\begin{theorem}\label{thm:conjugation}
Conjugation on $(a,b)$-cores coincides with conjugation on $(a,b)$-Dyck paths via Anderson's bijection:
\[
\core(\P)^c = \core(\P^c).
\]
\end{theorem}

This follows directly by showing the equivalence between the leading hooks of $\core(\P)^c$ and the positive hooks of $\P^c$.  A result of Olsson gives the leading hooks of $\core(\P)^c$; we include a proof for completeness.

\begin{lemma}\cite[Lemma 2.2]{Olsson93}\label{lem:conjugate_hooks}
 Let $\kappa$ be any partition with leading hooks given by the set $H$, with $m=\max(H)$. The conjugate partition $\kappa^c$ has leading hooks given by $\{m-n : n\in \{0,1,\dots ,m\}\setminus H\}$. 
\end{lemma}

\begin{proof}
Let $\kappa$ be any partition with leading hooks (hooks in the first column) given by the set $H$, with $m=\max(H)$. The leading hooks of its conjugate partition are the hooks in the top row of~$\kappa$. This partition has one column for each number $n$ in the set $\{0,1,\dots ,m\}\setminus H$. The hook of the upper box in the column corresponding to $n$ is equal to $m-n$ as illustrated in Figure~\ref{fig:conjugate_proof}.  
\end{proof}

\begin{figure}
  \begin{center}
  \includegraphics[width=0.9\textwidth]{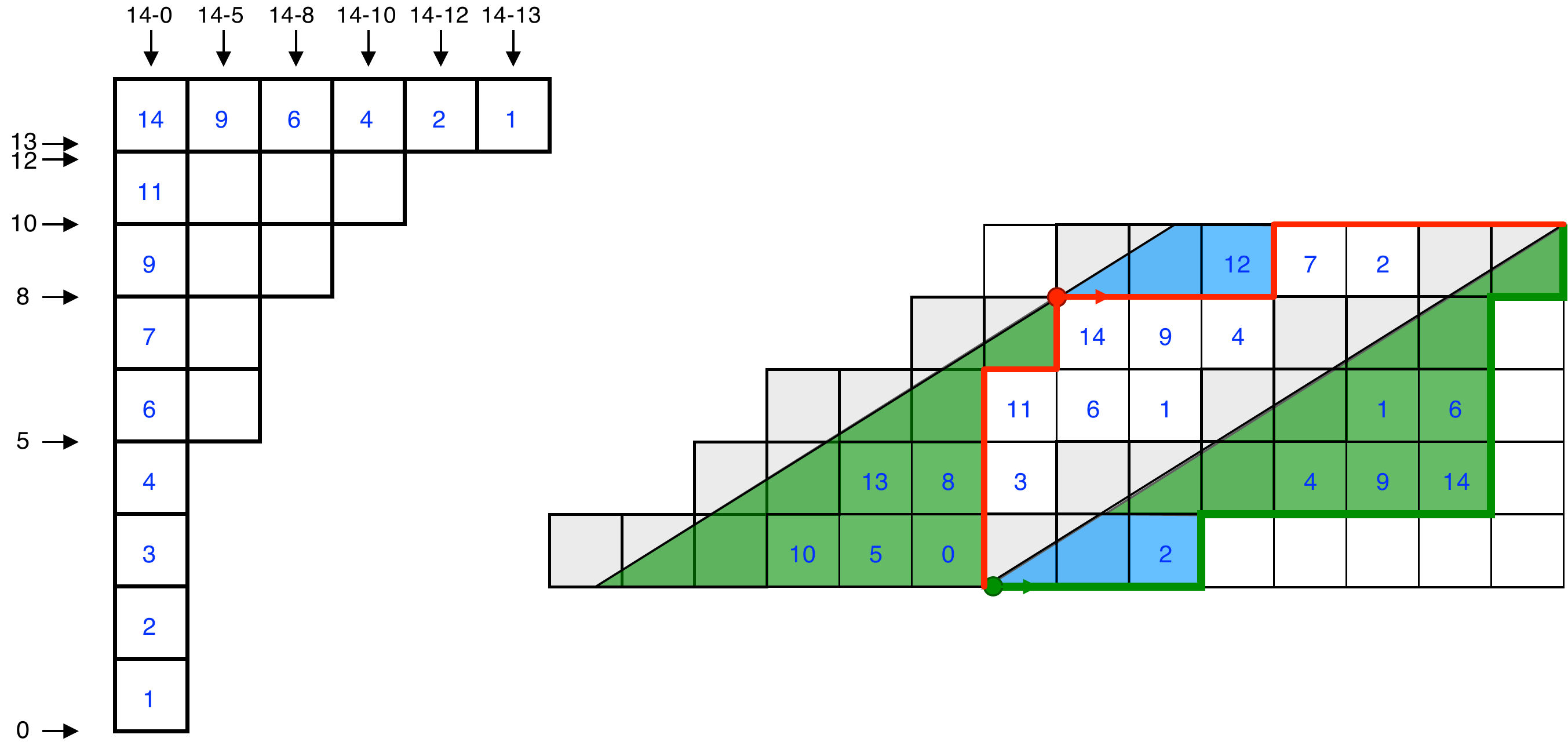}
  \end{center}
  \caption{Illustration of the proof of Lemmas~\ref{lem:conjugate_hooks} and \ref{lem:conjugate_positive_hooks}}
  \label{fig:conjugate_proof}
\end{figure}

\begin{lemma}\label{lem:conjugate_positive_hooks}
 Let $\P$ be an $(a,b)$-Dyck path with positive hooks given by $H$, with $m=\max(H)$. The conjugate path $\P^c$ has positive hooks given by $\{m-n : n\in \{0,1,\dots ,m\}\setminus H\}$.
\end{lemma}

\begin{proof}
Let $\P$ be an $(a,b)$-Dyck path with positive hook set given by $H$ and where $m=\max(H)$. Fill all the boxes on the left of the path with the hooks that are less than~$m$. Hooks appearing in the same row are equivalent mod $a$. Furthermore, the rows contain all the residues $0,1,\dots, a-1$ modulo $a$ because $a$ and $b$ are relatively prime and, as a consequence, the filled hooks contain all the numbers from 0 to $m$. 

Draw a diagonal parallel to the main diagonal passing through the upper left corner of the box  below $\P$ with the largest hook $m$. Consider the area $A$ below this diagonal directly on the left of $\P$ as illustrated in Figure~\ref{fig:conjugate_proof}. The boxes in $A$ are exactly the boxes on the left of the path with hook length~$n$ less than $m$. Applying cyclic shifts to $P$ to obtain a path below the main diagonal transforms the area $A$ to the area between the main diagonal and the shifted path. Since this transformation maps the box with hook length $m$ to the box with hook length 0 (when rotated~$180$ degrees), the hook length $n$ gets transformed to the hook length $m-n$.
\end{proof}

\begin{example}
In both Figure~\ref{fig:conjugate_cores} and Figure~\ref{fig:conjugate_paths}, the set of hooks on the left is $H=\{1,2,3,4,6,7,9,11,14\}$, with $m=14$. The set $\{0,1,\dots , m\}\setminus H=\{0,5,8,10,12,13\}$, and subtracting these numbers from~$14$ we get that the leading and positive hooks of the conjugate are $\{14,9,6,4,2,1\}$ as desired.
\end{example}

\begin{theorem}
\label{thm.slconj}
The skew length of $\P$ is equal to the skew length of $\P^c$.
\end{theorem}

\begin{proof}
Let $n_i>e_j$ be a skew inversion for the path $\P$, with largest level $m$. The north and east steps of the conjugate path are in correspondence with the north and east steps in the original path, respectively. The corresponding north and east levels are given by $n_i'=m-n_i-b$ and $e_j'=m-e_j+a$. A simple calculation shows that these satisfy $n_i'+b<e_j'-a$, giving a flip skew inversion for $\P^c$. Thus, there is a one-to-one correspondence between skew inversions for $\P$ and flip skew inversions in $\P^c$ (and a similar correspondence between flip skew inversions for $\P$ and skew inversions in $\P^c$). The result follows directly from Theorem~\ref{thm:skewLength-skewInversions}. 
\end{proof}

\begin{remark}\label{rem:conj_flip}
As explained in the proof of Theorem~\ref{thm.slconj} the number of skew inversions of $\P^c$ is equal to the number of flip skew inversions of $\P$.  Therefore, the skew length of a conjugate path may be thought of as the skew length of the original path when flipped to a $(b,a)$-Dyck path. 
\end{remark}

Consider the hook lengths of the boxes in the first row of an $(a,b)$-core partition $\kappa$.  Find the largest hook length of each residue modulo $a$.  The \emph{$a$-columns} of $\kappa$ are the columns of $\kappa$ corresponding to these hook lengths. Theorem~\ref{thm.slconj} implies the following result, which is illustrated in Figure~\ref{fig:conjugate_cores_skewlenght}.

\begin{corollary}\label{cor:slconj_cores}
Let $\kappa$ be an $(a,b)$-core partition. The number of boxes in the $a$-rows and $b$-boundary of $\kappa$ is equal to the number of boxes in the $a$-columns and $b$-boundary of $\kappa$.
\end{corollary}

\begin{proof}
The number of boxes in the $a$-rows and $b$-boundary of $\kappa$ is equal to the skew length of $\kappa$. The number of boxes in the $a$-columns and $b$-boundary of $\kappa$ is equal to the skew length of $\kappa^c$. The result then follows from Theorem~\ref{thm:conjugation} and Theorem~\ref{thm.slconj} by applying Anderson's bijection.
\end{proof}

\begin{figure}
  \begin{center}
  \includegraphics[scale=0.5]{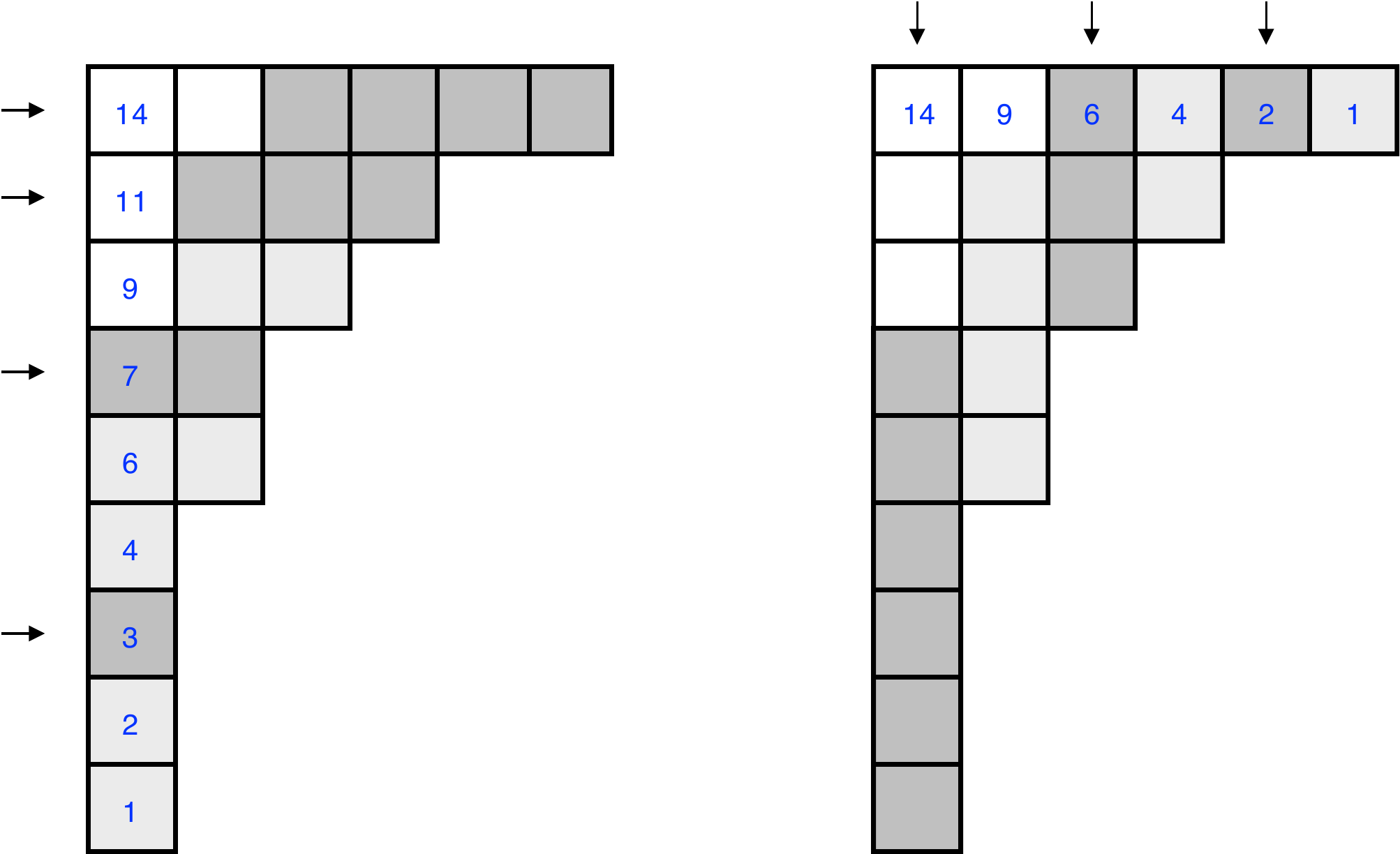}
  \end{center}
  \caption{(Left) The $8$-boundary boxes of our favorite $(5,8)$-core $\kappa$ are shaded; those in the $5$-rows of $\kappa$ are darker. 
  (Right) The $8$-boundary boxes of $\kappa$ are shaded; those in the $5$-columns of $\kappa$ are darker.
  The number of darkly shaded boxes on the left $4+3+2+1=10$ is equal to the number of darkly shaded boxes on the right~$6+3+1=10$. (See Corollary~\ref{cor:slconj_cores}.)}
  \label{fig:conjugate_cores_skewlenght}
\end{figure}

\section{The zeta map (and eta)}
\label{sec:zeta_map}

The zeta map is an intriguing map from $\DD_{a,b}$ to $\DD_{a,b}$ which can be defined in a wide variety of ways. See, for example, \cite{AHJ14,ALW-sweep,ALW-parking,GMII}, with equivalence of many definitions given in~\cite{ALW-sweep}.  The precise description of zeta depends on making some choices; in our experience, these choices always resolve into one of two distinct maps, which we call zeta and eta.  The eta map  can be interpreted as the zeta map applied to the conjugate of $\P$, as reproved in Proposition~\ref{prop:eta_zeta_conjugate} by appealing to skew inversions. The joint dynamics of zeta and eta will be used to present a combinatorial description of the inverse of zeta in Section~\ref{sec:inverse}.

In this section we present four combinatorial descriptions for computing the zeta and eta maps, starting with an interpretation involving core partitions implicit in \cite{AHJ14}, followed by with an equivalent description via the sweep maps considered in~\cite{ALW-sweep}.  Our main contributions are two new combinatorial descriptions of the zeta map involving interval intersections and a laser filling, along with the study of the eta map in all four contexts.  

\subsection{Zeta and eta via cores}\

Drew Armstrong conjectured a combinatorial interpretation for the zeta map by way of core partitions, drawing inspiration from Lapointe and Morse's bounded partitions \cite{LM}, after learning of Loehr and Warrington's sweep map discussed in the next section.  We present his definition and provide a parallel definition for the eta map.

\begin{definition}\label{def:zeta_eta}
Let $\P$ be an  $(a,b)$-Dyck path and let $\core(\P)$ be its corresponding $(a,b)$-core. From~$\P$ define two partitions $\lam(\P)$ and $\mu(\P)$ and corresponding lattice paths $\zeta(\P)$ and $\eta(\P)$:
\begin{itemize}
\item $\lambda(\P)=(\lambda_1, \dots , \lambda_a)$ is the partition that has parts equal to the number of $b$-boundary boxes in the $a$-rows of $\core(\P)$. 
\item $\mu(\P)=(\mu_1,\dots,\mu_b)$ is the partition that has parts equal to the number of  $a$-boundary boxes in the $b$-rows of $\core(\P)$.
\item $\zeta(\P)$ is the $(a,b)$-Dyck path that bounds the partition $\lambda(\P)$. 
\item $\eta(\P)$ is the $(a,b)$-Dyck path that bounds the conjugate of the partition $\mu(\P)$.  
\end{itemize}
The \emph{zeta map} $\zeta:\DD_{a,b}\row\DD_{a,b}$ is defined by $\zeta:\P\mapsto\zeta(\P)$.  The \emph{eta map} $\eta:\DD_{a,b}\row\DD_{a,b}$ is defined by $\eta:\P\mapsto\eta(\P)$.   
\end{definition}

One can see from the definition of zeta and eta, via the sweep map described below, that $\zeta(\P)$ and $\eta(\P)$ are indeed paths that stay above the main diagonal. We refer to~\cite{ALW-sweep} for a proof.

An alternative method for calculating $\lambda(\P)$ and $\mu(\P)$ follows from Lemmas~\ref{lem:skew_inversions_length_dif} and~\ref{lem:skew_inversions_length_dif_b}.

\begin{lemma}\label{lem:lambda_mu_inversions}
The entries of the partitions $\lambda(\P)$ and $\mu(\P)$ satisfy:
\begin{enumerate}[(i)]
\item $\lambda_i$ is the number of skew inversions of $\P$ involving the north level $n_i$.
\item $\mu_j$ is the number of flip skew inversions of $\P$ involving the east level $e_j$.
\end{enumerate}
\end{lemma}

In the $(n,n+1)$ case, the zeta map specializes to the map studied in~\cite{haglund2008q} for classical Dyck paths, which sends the $\dinv$ and $\area$ statistics considered by Haiman to the $\area$ and $\bounce$ statistics considered by Haglund. One of the main interests on the zeta map is the fact that it sends skew length to co-area, or equivalently, co-skew length to area. 

\begin{corollary}
The skew length of $\P$ is equal to the co-area of $\zeta(\P)$.
\end{corollary}

\begin{proof}
The co-area of $\zeta(\P)$ is by definition equal to the number of boxes in the partition $\lambda$. By Lemma~\ref{lem:lambda_mu_inversions}, this number of boxes counts the number of skew inversions of $\P$, and thus is equal to the skew length of $\P$.  
\end{proof}

\begin{remark}\label{rem:dinv}
The $\dinv$ statistic for classical Dyck paths can be generalized to the rational Catalan case as the number of boxes $B$ above the path satisfying 
\[
\frac{\text{arm}(B)}{\text{leg}(B)+1} \leq \frac{b}{a} < \frac{\text{arm}(B)+1}{\text{leg}(B)},
\]
where arm denotes the number of boxes directly on the right of $B$ above the path, and leg denotes the number of boxes directly below $B$ above the path. This intriguing statistic also satisfies~\mbox{$\dinv(P)=\area(\zeta(P))$}, see~\cite[Theorem~16]{loehr_continuous_2009} and~\cite{GMI}. As a consequence the co-skew length and dinv statistics are the same,
\begin{equation}
\sl'(P) = \dinv(P).
\end{equation}
Note that the definition of $\dinv$ is preserved by flipping an $(a,b)$-Dyck path to a $(b,a)$-Dyck path, and therefore skew length is preserved by flipping (as alternatively proved in Corollary~\ref{cor.absl}).
By Remark~\ref{rem:conj_flip}, the skew length of the conjugate of $\P$ is equal to the skew length of $\P$ when flipped to a $(b,a)$-Dyck path. This provides an alternative proof that skew length is preserved under conjugation (Theorem~\ref{thm.slconj}). 
\end{remark}

The work of Gorsky and Mazin~\cite[Theorem 8]{GMII} and of Armstrong, Loehr, and Warrington~\cite[Table~1]{ALW-sweep} include the following proposition; we present a new proof involving skew inversions.

\begin{proposition}[\cite{GMII,ALW-sweep}]\label{prop:eta_zeta_conjugate} Let $\P$ be an $(a,b)$-Dyck path.  Then
\[
\eta(\P) = \zeta(\P^c).
\]
\end{proposition}

\begin{proof}
There is a one-to-one correspondence between the skew inversions of $\P^c$ and the flip skew inversions of $\P$, as shown in the proof of Theorem~\ref{thm.slconj}.  Through Lemma~\ref{lem:lambda_mu_inversions}, one deduces that~$\lambda(\P^c)$ is the conjugate of $\mu(\P)$. As a consequence, $\zeta(\P^c)=\eta(\P)$.
%
\end{proof}

\begin{remark}
In~\cite{GMII}, conjugation is considered in terms of normalized dual semimodules. The zeta and eta maps correspond to the maps $G_{m,n}$ and $G_{n,m}$ in~~\cite[Section~2.3]{GMII}. 
\end{remark}

Denote by $\P^\flip$ the result of flipping an $(a,b)$-Dyck path $\P$ to a $(b,a)$-Dyck path. The example corresponding to the path $\P$ in Figure~\ref{fig:zetaEta_cores} and the following result are illustrated in Figure~\ref{fig:zetaEta_flip}. This result can also be essentially found in~\cite[Table~1]{ALW-sweep}.

\begin{proposition}[\cite{ALW-sweep}]
Let $\P$ be an $(a,b)$-Dyck path. Then,
\begin{align*}
 \zeta(\P^\flip) &= \eta(\P)^\flip, \\
 \eta(\P^\flip) &= \zeta(\P)^\flip. 
\end{align*}
\end{proposition}

\begin{proof}
The skew inversions of $\P^\flip$ are in correspondence with the flip skew inversions of $\P$, and therefore $\lambda(\P^\flip)=\mu(P)$. As a consequence, $\zeta(\P^\flip) = \eta(\P)^\flip$. A similar argument shows that~$\mu(\P^\flip)=\lambda(P)$ and $\eta(\P^\flip) = \zeta(\P)^\flip$.
\end{proof}

\begin{figure}
  \begin{center}
  \includegraphics[scale=0.55]{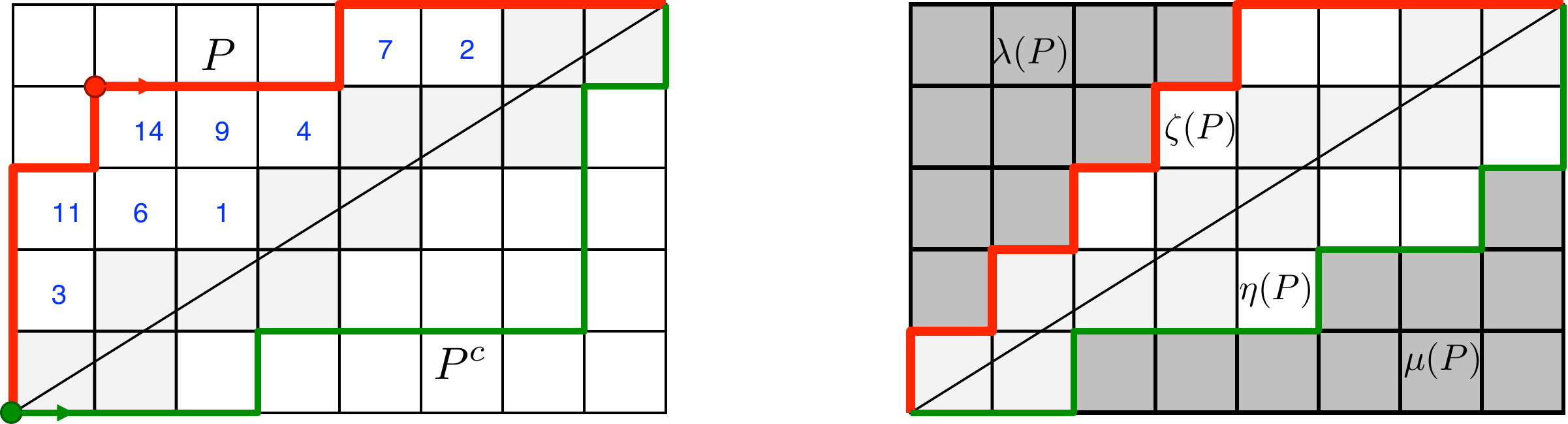}
  \end{center}
  \caption{In our running example, $\zeta(\P)$ bounds the partition $\lambda(\P)=(4,3,2,1,0)$ and $\eta(\P)$ bounds the conjugate of the partition $\mu(\P)=(3,2,2,1,1,1,0,0)$. }
  \label{fig:zetaEta_cores}
\end{figure}

\begin{example}
\label{ex:zetaEta_cores}
Figure~\ref{fig:zetaEta_cores} illustrates an example of the zeta map and the eta map applied to our running example path $\P$. From Example~\ref{ex:skew_inversions}, the $8$-boundary boxes in the $5$-rows of $\core(\P)$ give $\lambda(\P)=(4,3,2,1,0)$ and the $5$-boundary boxes in the $8$-rows of the core $\core(\P)$ give $\mu(\P)=(3,2,2,1,1,1,0,0)$.  Then $\zeta(\P)$ is the path that bounds $\lam(\P)$ and $\eta(\P)$ is the path that bounds the conjugate partition $\mu(\P)^c=(6,3,1,0,0)$.  We often combine $\lambda(\P)$, $\zeta(\P)$, $\mu(\P)$, and $\eta(\P)$ as on the right hand side of Figure~\ref{fig:zetaEta_intervals}.

The core partition $\core(\P^c)$ corresponding to the conjugate path $\P^c$ is illustrated in the right part of Figure~\ref{fig:conjugate_cores}. The $a$-rows of this core are the rows with leading hooks 14, 6, and 2. Counting the number of $b$-boundary boxes in these rows shows that $\lambda(\P^c)=(6,3,1,0,0)$, which equals $\mu(\P)^c$. We see that $\eta(\P)=\zeta(\P^c)$. 
\end{example}

\begin{figure}
  \begin{center}
  \includegraphics[scale=0.55]{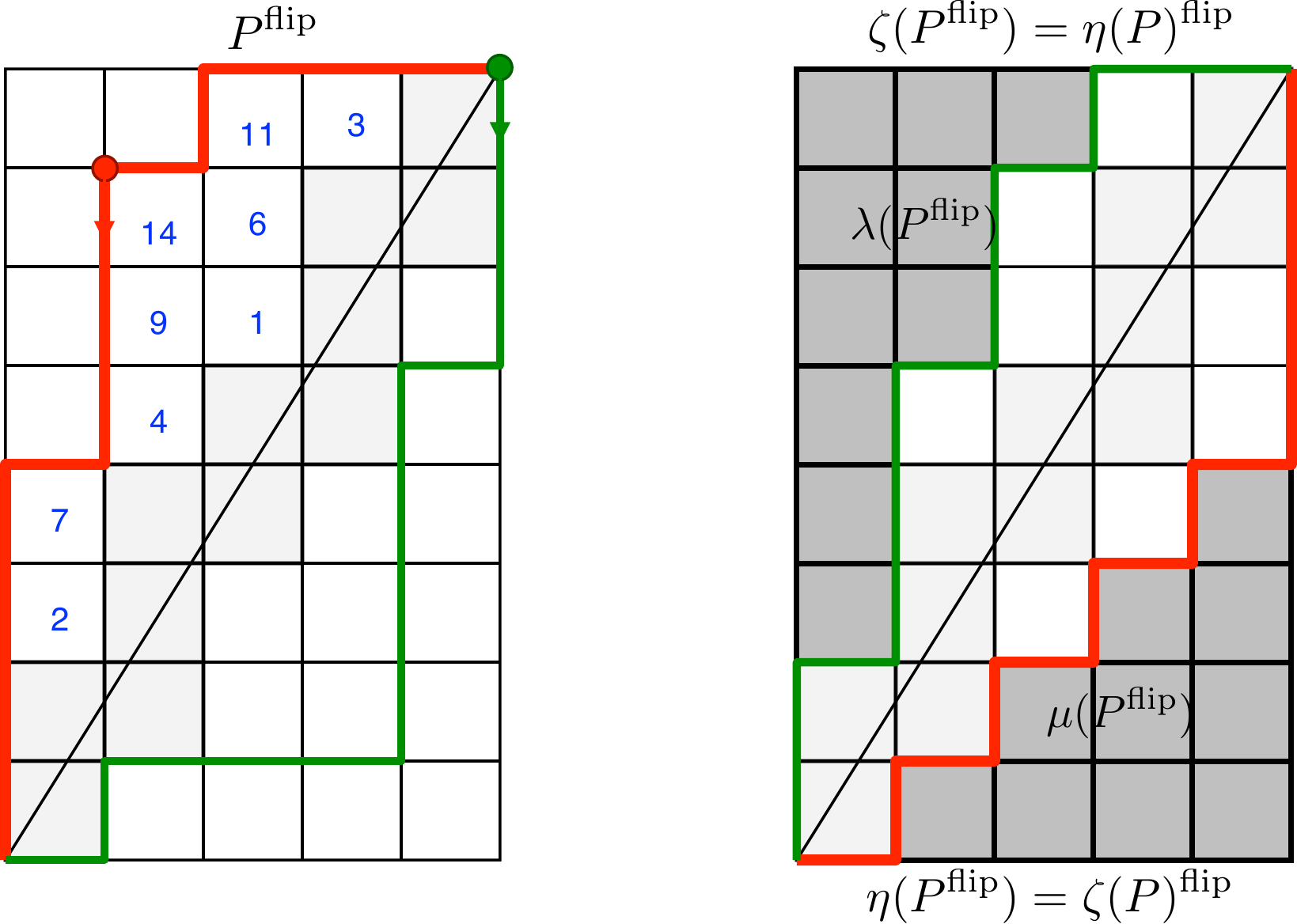}
  \end{center}
  \caption{Zeta and eta applied to the flipped Dyck path of our running example path $P$. }
  \label{fig:zetaEta_flip}
\end{figure}

\subsection{Zeta and eta via sweep maps}\label{sec:zetaEta_sweep}\

\renewcommand*{\thefootnote}{\fnsymbol{footnote}}
This section presents the combinatorial description of the zeta map on rational Dyck paths as a \emph{sweep map} created by Loehr and Warrington in~\cite{ALW-sweep}. 

Heuristically, this map `sweeps' the line of fixed slope $\frac{a}{b}$  across $\P$ starting on the main diagonal moving to the northwest, recording north and east steps in the order in which they are met. 
Analogously, the eta map `sweeps' the line of slope $\frac{a}{b}$ across $\P$ starting at the farthest point from the main diagonal moving to the southeast, recording south and west steps in the order in which they are met\footnote{These definitions exhibit the choice of `east-north' or `west-south' convention in~\cite{ALW-sweep}.}. This procedure is illustrated for our running example in Figure~\ref{fig:zetaEta_sweep}.

\begin{figure}
  \begin{center}
  \includegraphics[scale=0.55]{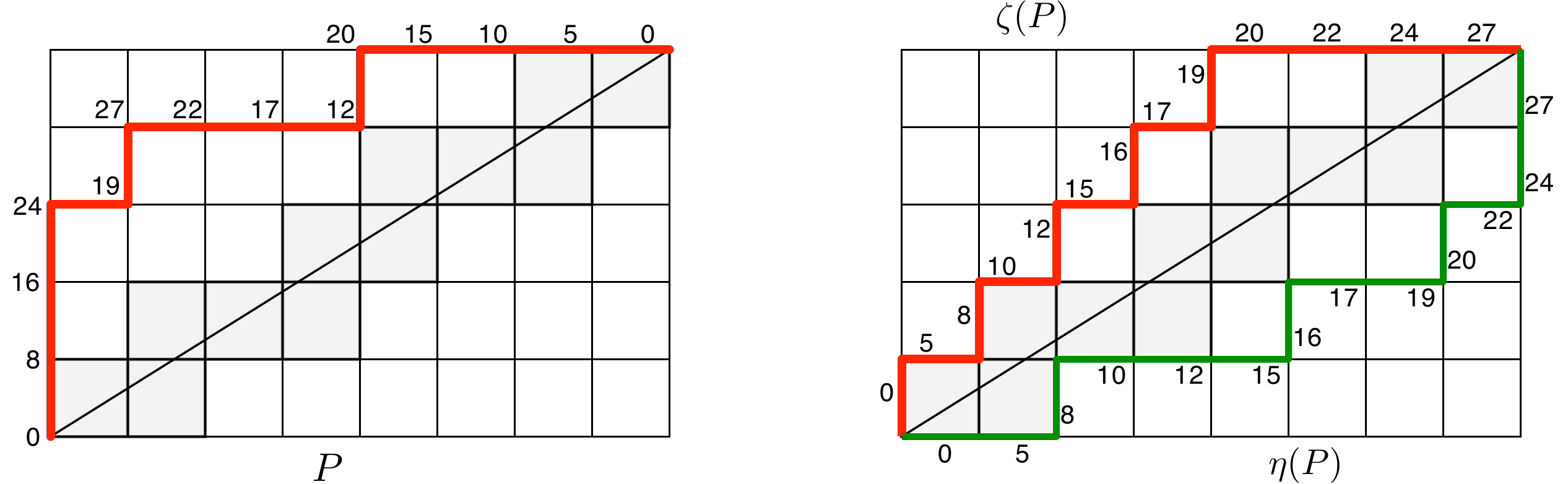}
  \end{center}
  \caption{Zeta and eta via sweep maps. 
  The steps of $\zeta(\P)$ are labeled by the levels of the lattice points of $\P$ in order, recording whether they correspond to north or east levels.
  The steps of $\eta(\P)$ are labeled by the levels of the lattice points of $\P$ in reverse order starting from the upper right corner, recording whether they correspond to south or west levels.}
  \label{fig:zetaEta_sweep}
\end{figure}

Recall that the reading word $L(\P)$ is obtained by reading the levels that occur along the path from southwest to northeast, excluding the final $0$, and the reverse reading word $M(\P)$ is obtained by reading from northeast to southwest, excluding the final $0$. 

\begin{theorem}[{\cite{ALW-sweep}}]\label{thm:zeta_sweep}
The zeta map can be computed as follows:
\begin{enumerate}
\item[$(1)$] Place a bar over each of the entries of $L(\P)$ corresponding to an east step; these occur exactly at the \emph{right (cyclic) descents} of $\sigma$.
\item[$(2)$] Sort $L(\P)$ in increasing order, keeping track of the bars on various values.
\item[$(3)$] Read the resulting sequence of labels (bars and non-bars) to produce a new northeast lattice path, which we denote $\zeta(\P)$.
\end{enumerate}
\end{theorem}

\begin{theorem}\label{thm:eta_sweep}
The eta map can be computed as follows:
\begin{enumerate}
\item[$(1')$] Place a bar over each of the entries of $M(\P)$ corresponding to a {west} step; these occur exactly at the \emph{right (cyclic) {ascents}} of $\tau$.
\item[$(2')$] Sort $M(\P)$ in increasing order, keeping track of the bars on various values.
\item[$(3')$] Read the resulting sequence of labels (bars and non-bars) to produce a new southwest lattice path from $(b,a)$ to $(0,0)$, which we denote $\eta(\P)$.
\end{enumerate}
\end{theorem}

\begin{example}
In our running example in Figure~\ref{fig:zetaEta_sweep}, we mark the reading word 
\[
L(\P)=(0,8,16,\bar{24},19,\bar{27},\bar{22},\bar{17},12,\bar{20},\bar{15},\bar{10},\bar{5}),
\] 
which sorts to $(0,\bar{5},8,\bar{10},12,\bar{15},16,\bar{17},19,\bar{20},\bar{22},\bar{24},\bar{27})$. Thus $\zeta(\P)$ is the path 
\[NENENENENEEEE.\]
We mark the reverse reading word
\[M(\P)=(\bar{0},\bar{5},\bar{10},\bar{15},20,\bar{12},\bar{17},\bar{22},27,\bar{19},24,16,8),\] 
which sorts to  
$(\bar{0},\bar{5},8,\bar{10},\bar{12},\bar{15},16,\bar{17},\bar{19},20,\bar{22},24,27)$.  Thus $\eta(\P)$ is the path 
\[WWSWWWSWWSWSS \textup{, which is equivalent to } NNENEENEEENEE.\]
\end{example}

\begin{remark}
Note that both computations in Theorem~\ref{thm:zeta_sweep} and Theorem~\ref{thm:eta_sweep} can be performed just as easily on the standardization $\sigma(\P)$ of $L(\P)$, since only the relative values of the labels matter.
\end{remark}

\begin{proof}[Proof of Theorems~\ref{thm:zeta_sweep} and~\ref{thm:eta_sweep}]
Consider the path $\zeta(\P)$ described in Theorem~\ref{thm:zeta_sweep}. The number of boxes on the left of the north step corresponding to a north level $n_i$ of $\P$ is equal to the number of east levels smaller that $n_i$. This number is equal to the number of skew inversions involving~$n_i$, which coincides with $\lambda_i$ by Lemma~\ref{lem:lambda_mu_inversions}~$(i)$. Therefore the described algorithm to compute $\zeta(\P)$ coincides with the definition of zeta in Definition~\ref{def:zeta_eta}.

Consider the (rotation of the) path $\eta(\P)$ described in Theorem~\ref{thm:eta_sweep}. The number of boxes below a given west step of $\P$ is equal to the number of south levels smaller than the corresponding west level. This number is equal to the number of flip skew inversions involving the corresponding level~$e_j$, which coincides with $\mu_j$ by Lemma~\ref{lem:lambda_mu_inversions}~$(ii)$. Therefore, the described algorithm to compute $\eta(\P)$ coincides with the definition of eta in Definition~\ref{def:zeta_eta}.
\end{proof}

\subsection{Zeta and eta via the laser filling}\

This section presents a new interpretation of zeta and eta that are read from a \emph{laser filling} in the boxes below the path $\P$ and above the main diagonal. Our main result in this section describes the partitions $\lambda$ and $\mu$ in terms of the laser filling. 
This result will be used in Section~\ref{sec:square} to give a new combinatorial description of the inverse of the zeta map in the square case without the use of bounce paths.  

Figure~\ref{fig:zetaEta_laser} illustrates the following definition.

\begin{figure}
  \begin{center}
  \includegraphics[scale=0.55]{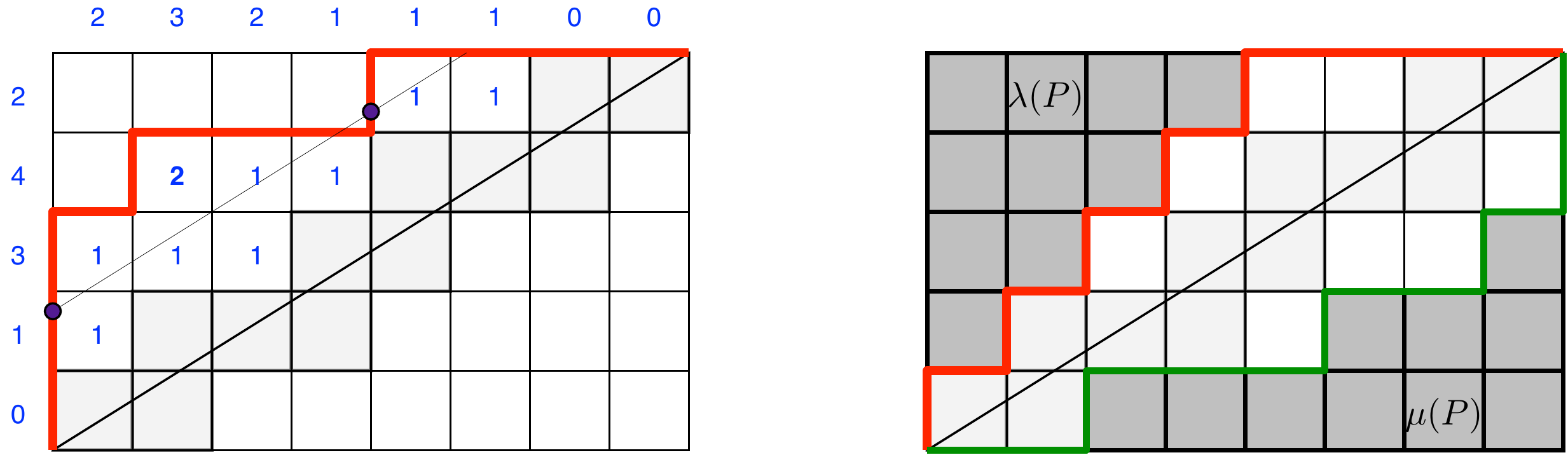}
  \end{center}
  \caption{Laser filling of a path $\P$. The laser pointed from the lower right corner of the box with filling 2 crosses two vertical walls of the path, while all other lasers cross only one. 
  The entries of the partition $\lambda(\P)=(4,3,2,1,0)$ are the sums of the laser fillings on the rows. The entries of the partition $\mu(\P)=(3,2,2,1,1,1,0,0)$ are the sums of the laser fillings on the columns. 
  }
  \label{fig:zetaEta_laser}
\end{figure}

\begin{definition}
Let $\P$ be an $(a,b)$-Dyck path and let $B$ be a box below $\P$ and above the line $y=\frac{a}{b}x$.  Draw the line of slope $\frac{a}{b}$ through the southeast corner of $B$ (a bi-directional laser).  The {\em laser filling} of $B$ is equal to the number of vertical walls of $P$ crossed by the laser. Equivalently, it is equal to the number of horizontal walls of $P$ crossed by the laser. 
\end{definition}

\begin{remark}
Lasers also appear in Armstrong, Rhoades, and Williams's \cite{ARW13}.  Their 
lasers stop at the first wall they meet; by contrast, our
lasers traverse (and count!)\ the walls of the path~$\P$.
\end{remark}

\begin{theorem}\label{thm:laser}
The partitions $\lambda$ and $\mu$ associated to $\P$ can be computed as follows:
\begin{enumerate}[(i)]
\item The parts of $\lambda(\P)$ are the sums of the laser fillings in the rows.
\item The parts of $\mu(\P)$ are the sums of the laser fillings in the columns.
\end{enumerate}
\end{theorem}

\begin{proof}
The entries of $\lambda$ count the skew inversions involving the north levels of each of the vertical steps in the path. For a given vertical step, this number is equal to the number of horizontal steps in the path that are strictly below the laser through its starting point. Each of these horizontal steps is crossed by exactly one of the lasers through the lower right corners of the boxes below the path that are in the same row of the vertical step in consideration. Statement {\em(i)} follows and Statement {\em (ii)} is proved similarly. 
\end{proof}

\begin{corollary}
The skew length of $\P$ is equal to the sum of the laser fillings of $\P$.
\end{corollary}

\begin{proof}
The skew length of $\P$ is equal to the area of $\lambda(\P)$. By the previous theorem, this area is equal to the sum of all laser fillings of $\P$.
\end{proof}

\subsection{Zeta and eta via interval intersections}\label{sec:zetaEta_intervals}\

This section presents a second new combinatorial interpretation of zeta and eta in terms of {\em interval intersections}.  Each step of the path $\P$ has an associated closed interval whose endpoints are the levels of its starting and ending points.  When the intervals are ordered in increasing order, zeta and eta can be directly determined.

\begin{definition}
\label{def:intervals}
Let $\P$ be an $(a,b)$-Dyck path.  Let $\N(\P)$ be the north levels of $\P$ and $\E(\P)$ be the east levels of $\P$.  Define the {\em north intervals} of $\P$ to be the set $\calI_N=\{[n_i,n_i+b] \text{ for } n_i\in\N(\P)\}$ and the {\em east intervals} of $\P$ to be the set $\calI_E=\{[e_j-a,e_j] \text{ for } e_j\in\E(\P)\}$.
\end{definition}

\begin{theorem}\label{cor:zetaEta_intervals}
Create an $a\times b$ grid.  Label the rows of the grid by the north intervals of $\P$ increasing from bottom to top, and the columns of the grid by the east intervals of $\P$ increasing from left to right.  Fill in the boxes in this grid when the corresponding row and column intervals do not intersect.  The boundary path of the shaded boxes above the main diagonal is $\zeta(\P)$ and the boundary path of the shaded boxes below the main diagonal is $\eta(\P)$, rotated 180 degrees.
\end{theorem}

\begin{proof}
This theorem is a straightforward consequence of Lemma~\ref{lem:lambda_mu_inversions}.
\end{proof}

\begin{figure}
  \begin{center}
  \includegraphics[scale=0.55]{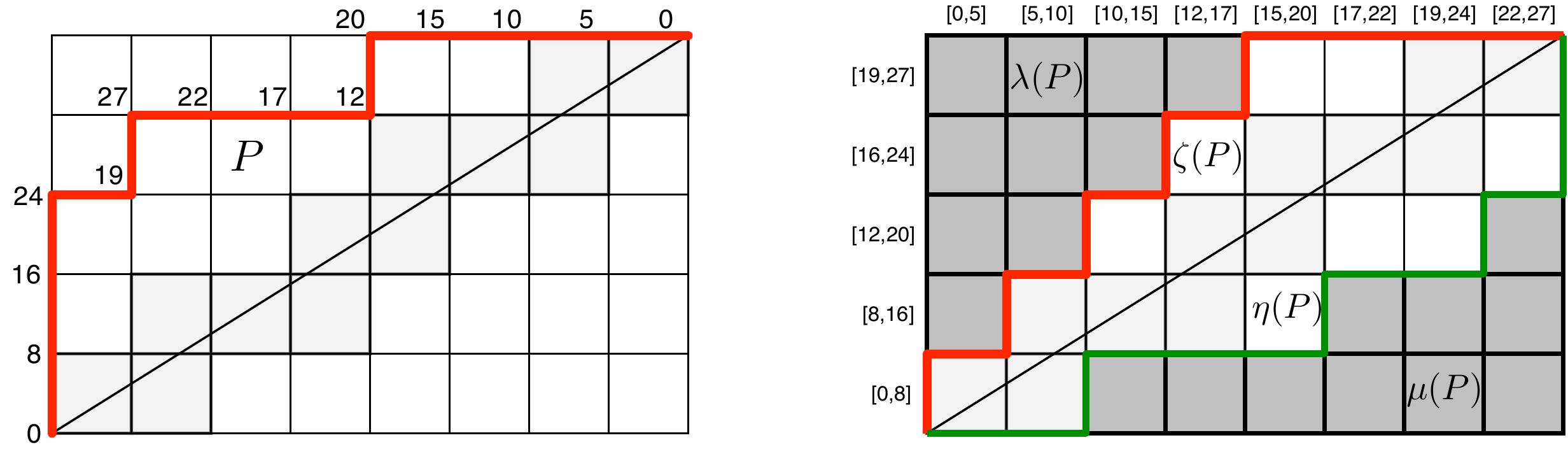}
  \end{center}
  \caption{Zeta and eta via interval intersections. The intervals on the left correspond to the ordered level intervals of the vertical steps in the path. The intervals on the top correspond to the level intervals of the horizontal steps. The shaded boxes of $\lambda$ and $\mu$ are the boxes whose corresponding row and column intervals do not intersect.}
  \label{fig:zetaEta_intervals}
\end{figure}

\begin{example}
  For our running example path $\P$, the north intervals are $[0,8]$, $[8,16]$, $[12,20]$, $[16,24]$, and $[19,27]$, which can be read directly from the north steps of $\P$, or calculated from the north levels as in Definition~\ref{def:intervals}.  Similarly, the east intervals of $\P$ are $[0,5]$, $[5,10]$, $[10,15]$, $[12,17]$, $[15,20]$, $[17,22]$, $[19,24]$, and $[22,27]$.  Labeling the rows of a $5\times 8$ grid with the north levels and the columns with the east levels gives the right side of Figure~\ref{fig:zetaEta_intervals}.  The shaded boxes are the those where the corresponding row interval does not intersect the corresponding column interval, from which $\lambda(\P)$, $\mu(\P)$, $\zeta(\P)$, and $\eta(\P)$ can be read posthaste.
\end{example}

\section{Pairing the zeta map with the eta map}\label{sec:inverse}

By considering the zeta map together with the eta map, we gain two new ideas: a new approach for proving that the zeta map is a bijection and (if $\zeta$ is a bijection) a new area-preserving involution on the set of $(a,b)$-Dyck paths.  For clarity and consistency, we have decided to use the letter $P$ to denote a path that is in the domain of $\zeta$ and use the letter $Q$ to denote a path that is in the image of $\zeta$.

\subsection{Inverse of the zeta map knowing eta}\label{sec:zetaeta}\

 For the image $\calZ$ under the pair of maps \[(\zeta,\eta):\DD\rightarrow\DD\times\DD,\] we define a map $\iota:\calZ\rightarrow\DD$ such that $(\zeta,\eta)\circ\iota$ is the identity map.  
Further, we conjecture that for every $(a,b)$-Dyck path $\Q$ that appears as the image of $\zeta$, there exists a unique $(a,b)$-Dyck path $\R$ such that $(\Q,\R)\in \calZ$.  This would imply that in~$\calZ$ every element of $\DD$ appears exactly once as the initial entry in the pair, from which it would follow that the zeta map is a bijection.

\begin{definition} 
A pair of $(a,b)$-Dyck paths $(\Q,\R)$ is an {\em admissible pair} if $(\Q,\R)=\big(\zeta(\P),\eta(\P)\big)$ for some $(a,b)$-Dyck path $\P$.  The {\em set of admissible pairs} $\calZ\subset\DD\times\DD$ is the image under the pair of maps $(\zeta,\eta):\DD\rightarrow\DD\times\DD$.  
\end{definition}

We now describe a simple combinatorial description of the inverse map $\iota$ that recovers $\P$ from the pair~$(\Q,\R)$ or, equivalently, from the pair of partitions ($\lambda$, $\mu$) they bound.  

\begin{definition} \label{def:iota}
Let $(\Q,\R)$ be an admissible pair.  Define $\iota(\Q,\R)$ as follows.
\begin{enumerate}
\item[(1)] Draw the path $\Q$ above the diagonal and rotate the path $\R$ 180 degrees so that it embeds below the diagonal in the same diagram.  Label the steps of each path from~$1$ to~$a+b$ starting at the bottom-left corner and ending at the top-right corner in the order in which they appear in the path. 
\item[(2)] Create the permutation $\gamma:[a+b]\to [a+b]$ as follows. If $l$ is a label of a horizontal step in $\Q$, define $\gamma(l)$ to be the label of the horizontal step in $R$ that is in the same column of~$l$.  If $l$ is a label of a vertical step in $\Q$, define $\gamma(l)$ to be the label of the vertical step in $R$ that is in the same row of $l$.
\item[(3)] For admissible pairs $(\Q,\R)$, $\gamma$ is a cycle permutation.  Interpret $\gamma$ in cycle notation as $(\sigma_1,\sigma_2,\hdots,\sigma_{a+b})$, fixing $\sigma_1=1$.
Define $\P=\iota(\Q,\R)$ to be the path whose east steps correspond to the cyclic descents of $\sigma$.\footnote{A descent occurs when $\sigma_i>\sigma_{i+1}$.  A cyclic descent is defined in the same way, but considering the indices modulo~$a+b$, allowing a descent in the last position of $\sigma$.}
\end{enumerate}
\end{definition}

\begin{theorem}
\label{thm:inverse}
$\gamma$ is a cycle permutation and the map $\iota$ is the inverse map for the pair $(\zeta,\eta)$.
\end{theorem}

\begin{proof}
Suppose $(\Q,\R)$ is an admissible pair, so that there exists a $\P\in\DD$ such that $(\Q,\R)=(\zeta(\P),\eta(\P))$.  Label the steps of $\Q$ and $\R$ with the levels of $\P$ as determined by the sweep map algorithm given in Theorems~\ref{thm:zeta_sweep} and \ref{thm:eta_sweep} (as illustrated in Figure~\ref{fig:zetaEta_sweep}).  The definition of the permutation $\gamma$ using these labels instead of on $[a+b]$ induces a permutation on the set of levels of the lattice points of $\P$. We will prove that this permutation is the cycle permutation given by the reading word $L(\P)$ of $\P$. 

Because of the relationship between the forward reading word $L(\P)$ and the reverse reading word $M(\P)$, the labels of the vertical steps of $\R$ are exactly the labels of the vertical steps of $\Q$ plus $b$, while the labels of the horizontal steps of $\R$ are exactly the labels of the horizontal steps of $\Q$ minus $a$. This implies that the permutation $\gamma$ maps the level of a lattice point in $\P$ to the level of the next lattice point along $\P$, forming a permutation on the set of labels that is a cycle ordered by the reading word~$L(\P)$. 

Since the level labels appear in order as we walk along $\Q$, only the relative order of the labels matters; returning all labels to the numbers from~$1$ up to $a+b$ recovers $\gamma(\P)$, which when interpreted as a permutation in one line notation is the reading permutation $\sigma(\P)$.  By Remark~\ref{rem:sigma}, we recover $\P$ directly from $\sigma(\P)$ and the result follows. 
\end{proof}

Taken with Theorem~\ref{thm:inverse}, the following conjecture would imply that $\zeta$ is a bijection.

\begin{conjecture}
Suppose that $\Q\in\DD_{a,b}$.  There exists at most one $\R\in\DD_{a,b}$ such that $(\Q,\R)\in\calZ$.
\end{conjecture}

\begin{example}
Figure~\ref{fig:zetaEta_inverse} illustrates  the procedure outlined in Definition~\ref{def:iota} for the pair~$(\Q,\R)=(\zeta(\P),\eta(\P))$ from our running example $\P$. After labeling the paths $\Q=\zeta(\P)$ and $\R=\eta(\P)$ from $1$ to $13$, we see that $\gamma(1)=3$, $\gamma(2)=1$, $\gamma(3)=7$, etc.  Writing $\gamma$ in cycle notation gives \[\gamma=(1,3,7,{\bf 12},9,{\bf 13},{\bf 11},{\bf8},5,{\bf10},{\bf6},{\bf4},{\bf2}).\]  If we instead interpret this sequence of numbers as the one line notation of a permutation $\sigma$, the cyclic descents of $\sigma$ are bolded and correspond to the east steps of $\iota(Q,R)$. We see that $\iota(Q,R)=\P$.
\begin{figure}
  \begin{center}
  \includegraphics[scale=0.5]{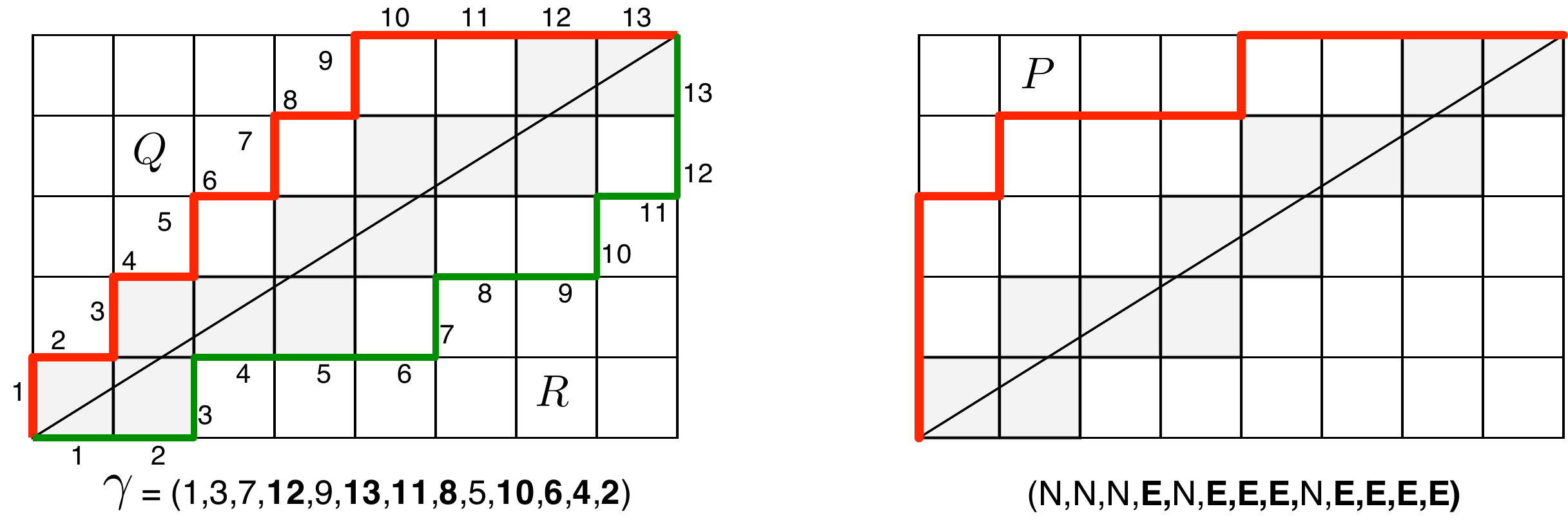}
  \end{center}
  \caption{Calculating $\P=\iota(\Q,\R)$ using the method in Definition~\ref{def:iota}.}
  \label{fig:zetaEta_inverse}
\end{figure}
\end{example}

\begin{remark}
The essence of the proof of Theorem~\ref{thm:inverse} is that the $\zeta$ and $\eta$ maps track the positions of the right cyclic descents of $L(\P)$ and  $M(\P)$.  Using these two sets of data, and the precise relationship between $L(\P)$ and $M(\P)$, we are able to solve for the levels of $\P$.  Interestingly, $\zeta(\P)$ does not obviously contain enough information to reconstruct $\P$. We cannot construct a unique permutation solely from its collection of right descents, and need additional information to recover $\P$.  In the standard Catalan case, this additional information is essentially implied by the particular structure of the $n\times (n+1)$ rectangle; for the general case, we obtain the extra information necessary from $\eta(\P)$.
\end{remark}

\begin{remark}
When pairing arbitrary $\Q$ and $\R$ paths a number of things can go wrong.  First, Theorem~\ref{thm.slconj} implies that in order to come from an actual path, we must have $\area(\Q)=\area(\R)$.  Second, we know that $\gamma$ must have a single cycle; it is simple to construct examples where this does not occur.  It is also possible to find pairs $(\Q,\R)$ where $\gamma$ has a single cycle, but the labels $l_i$ obtained from the reverse bijection are in the wrong relative order.  In other words, we may have $\zeta(\iota(\Q,\R))\neq \Q$. 
\end{remark}

We propose the problem of characterizing all possible permutations $\gamma(P)$. As a straightforward consequence of the description of this permutation in terms of the pair $Q$ and $R$, we conclude Proposition 6.8 without proof.

\begin{proposition}
\label{prop.exc}
The positions of the exceedences of $\gamma(\P)$ give the collection of north steps in~$\zeta(\P)$, and the values of the exceedences of $\gamma(\P)$ are the north steps in~$\eta(\P)$ when rotated $180^\circ$.
\end{proposition}

\subsection{An area-preserving involution on rational Dyck paths}\
\label{sec:perp}

If $\zeta$ is invertible, we can use $\eta$ to define a new area-preserving involution on the set of $(a,b)$-Dyck paths, induced by the conjugate map under~$\zeta$ which we call the conjugate-area map.  This involution sends the path $\zeta(\P)$ to the path $\eta(\P)=\zeta(\P^c)$ and is predictable for certain families of $(a,b)$-Dyck paths. 

\begin{definition}
The \emph{conjugate-area map} 
applied to an $(a,b)$-Dyck path $\Q$ is the path \[\chi(\Q):=\zeta \circ c \circ \zeta^{-1}(\Q).\] 
If $\lambda$ is the partition bounded by $\Q$, we define $\chi(\lambda)$ to be the partition bounded by $\chi(\Q)$.
\end{definition}

\begin{figure}[h]
\begin{tikzpicture}
  \matrix (m) [matrix of math nodes, row sep=4em, column sep=6em]
    { \P & \zeta(\P)   \\
     \P^c & \eta(\P)  \\ };
  { [start chain] \chainin (m-1-1);
  \chainin (m-1-2) [join={node[above,labeled] {\text{zeta}}}]; 
  \chainin (m-2-2) [join={node[right,labeled] {\text{conj-area}}}]; }
  { [start chain] \chainin (m-1-1);
  \chainin (m-2-1) [join={node[left,labeled] {\text{conjugate}}}];
  \chainin (m-2-2) [join={node[below,labeled] {\text{zeta}}}]; }
  { [start chain] \chainin (m-1-1);
  \chainin (m-2-2) [join={node[right,labeled] {\raisebox{.1in}{\text{\tiny eta}}}}]; }
\end{tikzpicture}
\caption{Diagrammatic description of the conjugate-area involution.}
\label{fig:alpha}
\end{figure}
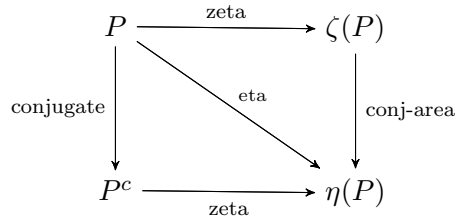

\begin{remark}
For partitions $\lambda$ and $\mu$ bounded by $\zeta(\P)$ and $\eta(\P)$ we have
$\chi(\lambda) = \mu^c$.
\end{remark}

\begin{proposition}
If the zeta map is a bijection then the conjugate-area map is an area-preserving involution on the set of $(a,b)$-Dyck paths.
\end{proposition}
\begin{proof}
Since conjugation is an involution, we see that applying the operator $\zeta \circ c \circ \zeta^{-1}$ twice is equal to the identity, and therefore $\chi(\chi(\Q)) = \Q$. Furthermore, conjugation preserves skew length~(Theorem~\ref{thm.slconj}), which is mapped to co-area via the zeta map. Thus, $\chi$ must be an area-preserving involution.
\end{proof}

One possible approach to prove that $\zeta$ is a bijection would be to directly construct the involution~$\chi$.  In Section~\ref{sec:square} we show that in the square case $\chi$ is exactly the map that reverses the path~$\P$; equivalently one finds $\chi(\lambda)$ by simple conjugation.  In the rational case, conjugation must fail in general because conjugates of partitions may not sit above the main diagonal.  Although, Proposition~\ref{prop:justified} exhibits our empirical observation that for `small' partitions $\lambda$, $\chi(\lambda)$ is often the conjugate.

We have found that $\chi$ is predictable in other families of examples as well; in Section~\ref{sec:inductive_zeta_inverse} we present an inductive combinatorial description of the inverse of the zeta map and of the area-preserving involution for a nice family of examples.

\begin{example}[Left-justified and up-justified partitions]
Consider two families of partitions whose Young diagrams fit above the main diagonal in the $a\times b$ grid.  Let $n\in \mathbb{N}$ be a number no bigger than the number of boxes above the main diagonal in the $a\times b$ grid.  Define the {\em left-justified partition}~$\lambda^n$ to be the unique partition whose Young diagram has $n$ boxes as far to the left as possible and the {\em up-justified} partition $\nu^n$ to be the unique partition whose Young diagram has $n$ boxes as far up as possible.  Figure~\ref{fig:leftDown_partitions} shows $\lam^8=(3,2,2,1)$ embedded above the diagonal and $\nu^8=(6,2)$ rotated 180 degrees and embedded below the diagonal.  We use the notation $\nu^n$ because it is the conjugate of what one might expect if we called it $\mu^n$, as pointed out by an astute referee.

\begin{proposition}
\label{prop:justified}
The left-justified and up-justified partitions are related by the conjugate-area map:
\[
\chi(\lambda^n) = \nu^n. 
\]
Moreover, $\zeta^{-1}(\lambda^n)$ is the path with area $n$ containing the first $n$ positive hooks in the grid.
\end{proposition}
\begin{proof}
Let $\P^n$ be the path containing the first $n$ positive hooks in the grid. This path consists of all the boxes below a line parallel to the main diagonal sitting in the highest level of the path, and therefore all the labels in the laser filling are equal to 1. Adding the labels in the rows and the columns we obtain the partitions $\lambda^n$ and $(\nu^n)^c$.
\end{proof}

\begin{figure}
  \begin{center}
  \includegraphics[scale=0.5]{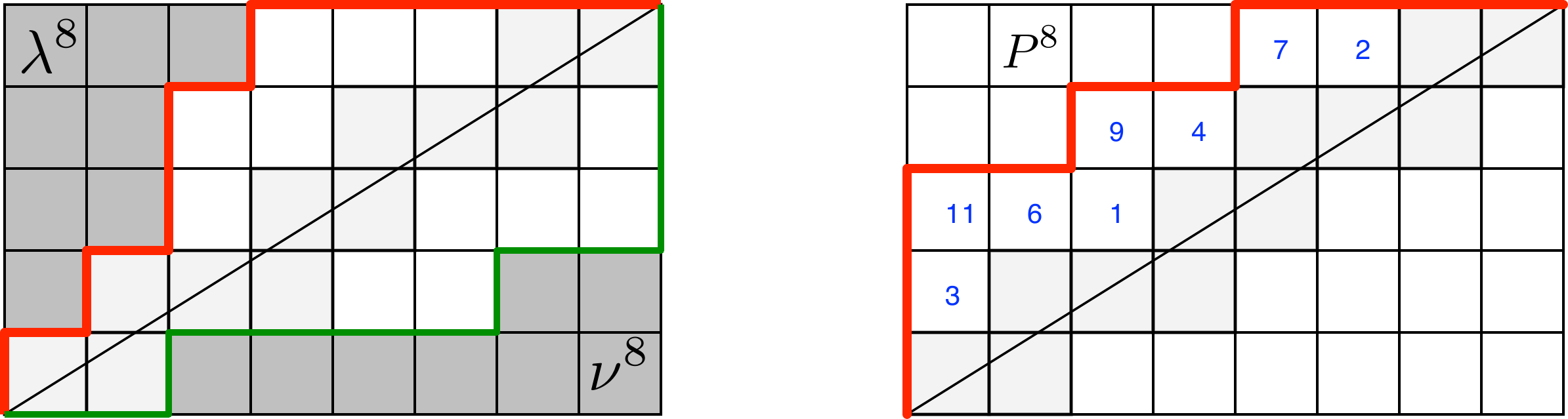}
  \end{center}
  \caption{The left-justified partition $\lambda^8$, up-justified partition $\nu^8$, and corresponding path~$\P^8$.}
  \label{fig:leftDown_partitions}
\end{figure}

Figure~\ref{fig:leftDown_partitions} illustrates and example of left-justified and up-justified partitions $\lambda^n$ and $\nu^n$ together with their corresponding path $\P^n$ for $n=8$. The reader is invited to verify that $\zeta(\P^8)$ and $\eta(\P^8)$ are given by the paths bounding $\lambda^8$ and $\nu^8$ using any of the methods described in Section~\ref{sec:zeta_map}, as well as to verify that the inverse map $\iota$ presented in Section~\ref{sec:inverse} gives $\P^8$ when applied to the paths bounding $\lambda^8$ and $\nu^8$.
\end{example}

\section{The square case}\label{sec:square}

In this section, we consider $(n,n+1)$-Dyck paths, lattice paths in an $n\times (n+1)$ grid staying above the main diagonal.  They are in bijection with classical Dyck paths in an $n\times n$ grid by simply forgetting the last east step of the path.
Haglund and Haiman~\cite{haglund2008q} discovered a beautiful description of the inverse of the zeta map in this case using a bounce path that completely characterizes the area sequence below the path. We present a new combinatorial description of the inverse of the zeta map in this case in terms of an area-preserving involution. This approach opens a new direction in proving that the zeta map is a bijection in the general $(a,b)$ case. 

\subsection{The conjugate-area involution, conjugate partitions and reverse paths.}\

Let $Q$ be an $(n,n+1)$-Dyck path. The area-preserving involution $\chi$ conjugates the partition~$\lambda$ bounded by the path $Q$.  This was proved in \cite[Theorem 9]{GMII}; we provide a new proof using our laser interpretation of zeta and eta.  For simplicity, denote by ${\Q}^r$ the path whose bounded partition is $\lambda^c$. We refer to $P^r$ as the \emph{reverse path} of $\Q$. Forgetting the last east step of the path, the reverse operation acts by reversing the path in the $n\times n$ grid. An example of the conjugate-area involution, conjugate partition and reverse path is illustrated in Figure~\ref{fig:square_perp}. 

\begin{figure}
  \begin{center}
  \includegraphics[scale=0.5]{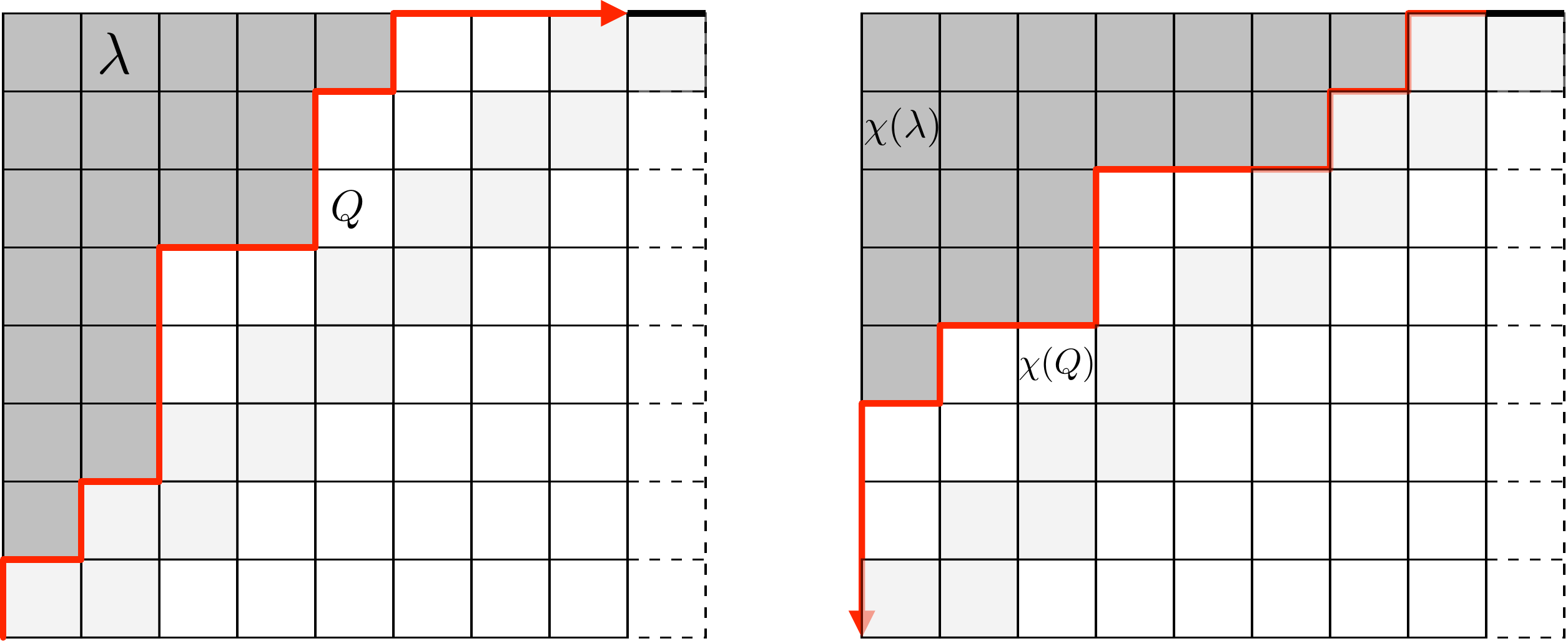}
  \end{center}
  \caption{The conjugate-area involution in the $(n,n+1)$ case. }
  \label{fig:square_perp}
\end{figure}

\begin{theorem}[\cite{GMII}]\label{thm:square_perp}
For a Dyck path $\Q$ and the partition $\lambda$ it bounds, we have $\chi(\Q) = \Q^r$ and~$\chi(\lambda) = \lambda^c$. 
\end{theorem}

\begin{proof}
We need to show that the partitions $\lambda$ and $\mu$ bounded by the images $\zeta(\P)$ and $\eta(\P)$ of any $(n,n+1)$-Dyck path $\P$ satisfy 
\[
\chi(\lambda) = \mu^c = \lambda^c. 
\]
Equivalently, we need to show that $\lambda=\mu$. 
The entries of the partitions $\lambda$ and $\mu$ are the sums of the labels in the laser filling of $\P$ over the rows and columns respectively (Theorem~\ref{thm:laser}). We will show that the values of the sums over the rows are in correspondence with the values of the sums over the columns, and therefore $\lambda=\mu$. This correspondence is illustrated for an example in Figure~\ref{fig:square_perp_proof}.

\begin{figure}
  \begin{center}
  \includegraphics[scale=0.5]{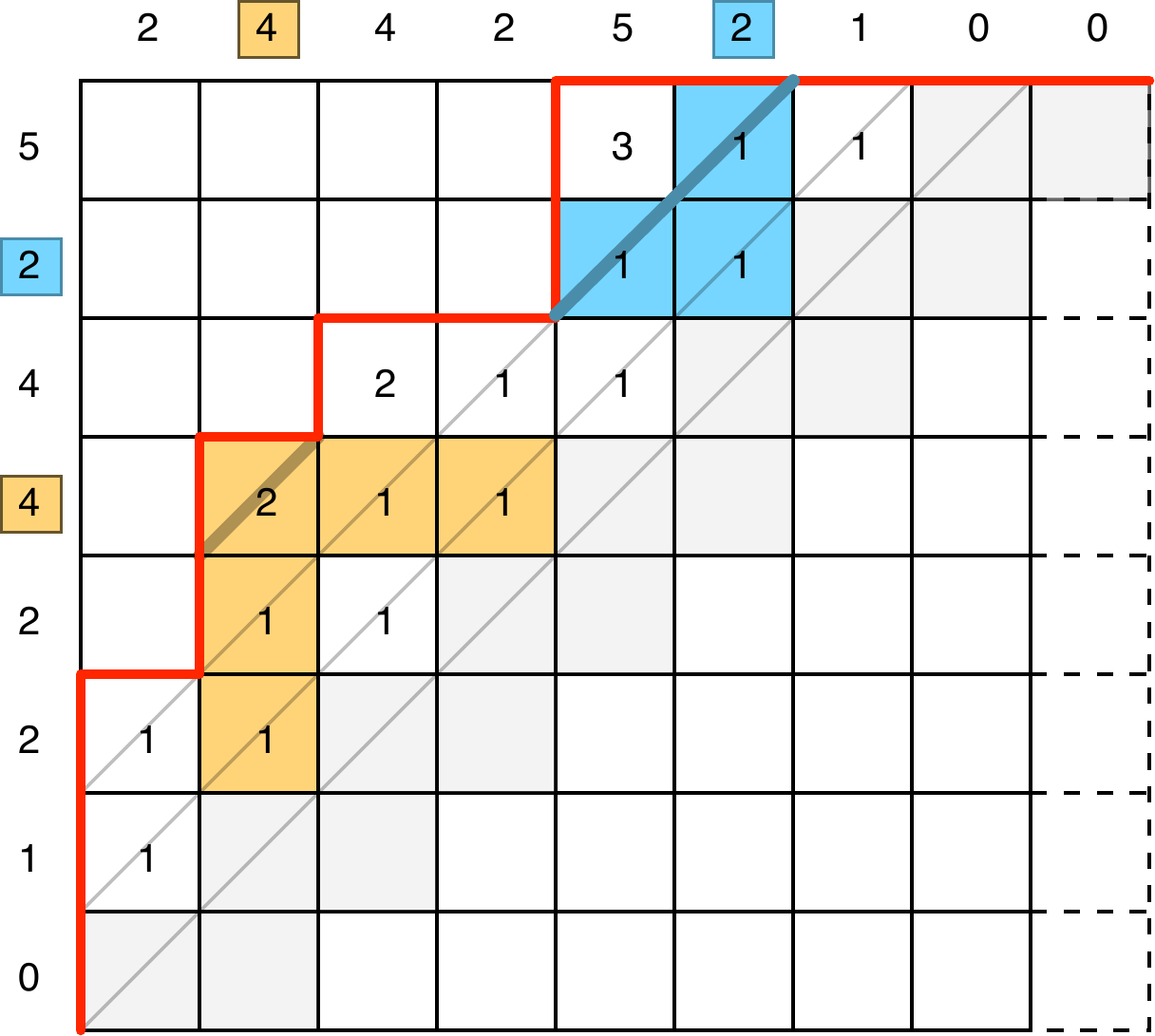}
  \end{center}
  \caption{Argument in the proof of Theorem~\ref{thm:square_perp}. The values of the sums over the rows are in correspondence with the values of the sums over the columns, and therefore $\lambda=\mu$.}
  \label{fig:square_perp_proof}
\end{figure}

For every row, draw a line of slope 1 in the northeast direction pointing from the starting point of the north step in that row. This line hits the path for the first time in the ending point of an east step of the path. The labels of the laser filling in the boxes in the column corresponding to this east step are exactly the same as the labels of the laser filling in the row in consideration.  (This is because the lasers are lines with slope $\frac{n}{n+1}$, which implies that for any two boxes on the same diagonal of slope 1 that are not interrupted in line of sight by the path $\P$, they will have the same laser filling.)  Thus, their corresponding sums are equal.  Doing this for all the rows gives the desired correspondence between the entries of the partition $\lambda$ and the entries of the partition $\mu$.  
\end{proof}

\subsection{The inverse of the zeta map}\

Because Theorem~\ref{thm:square_perp} provides the explicit formula for $\chi$, the method to find inverse of the zeta map in the $(n,n+1)$ case follows as a direct consequence of Theorem~\ref{thm:inverse}. The description of the map $\iota$ is presented in Definition~\ref{def:iota}.

\begin{theorem}\label{thm_square_inverse}
Let $\Q$ be an $(n,n+1)$-Dyck path. Then, $\zeta^{-1}(\Q)=\iota(\Q,\Q^r)$.
\end{theorem}

An example of this result is illustrated in Figure~\ref{fig:square_zeta_inverse1}. The laser filling of the path $\zeta^{-1}(\Q)$ in this example is shown in Figure~\ref{fig:square_perp_proof}. One can verify that the sum of the labels of the laser filling on the rows and columns gives rise to the partitions $\lambda$ and $\mu$ bounded by $\Q$ and $\chi(\Q)$ (Theorem~\ref{thm:laser}).

\begin{figure}
  \begin{center}
  \includegraphics[scale=0.5]{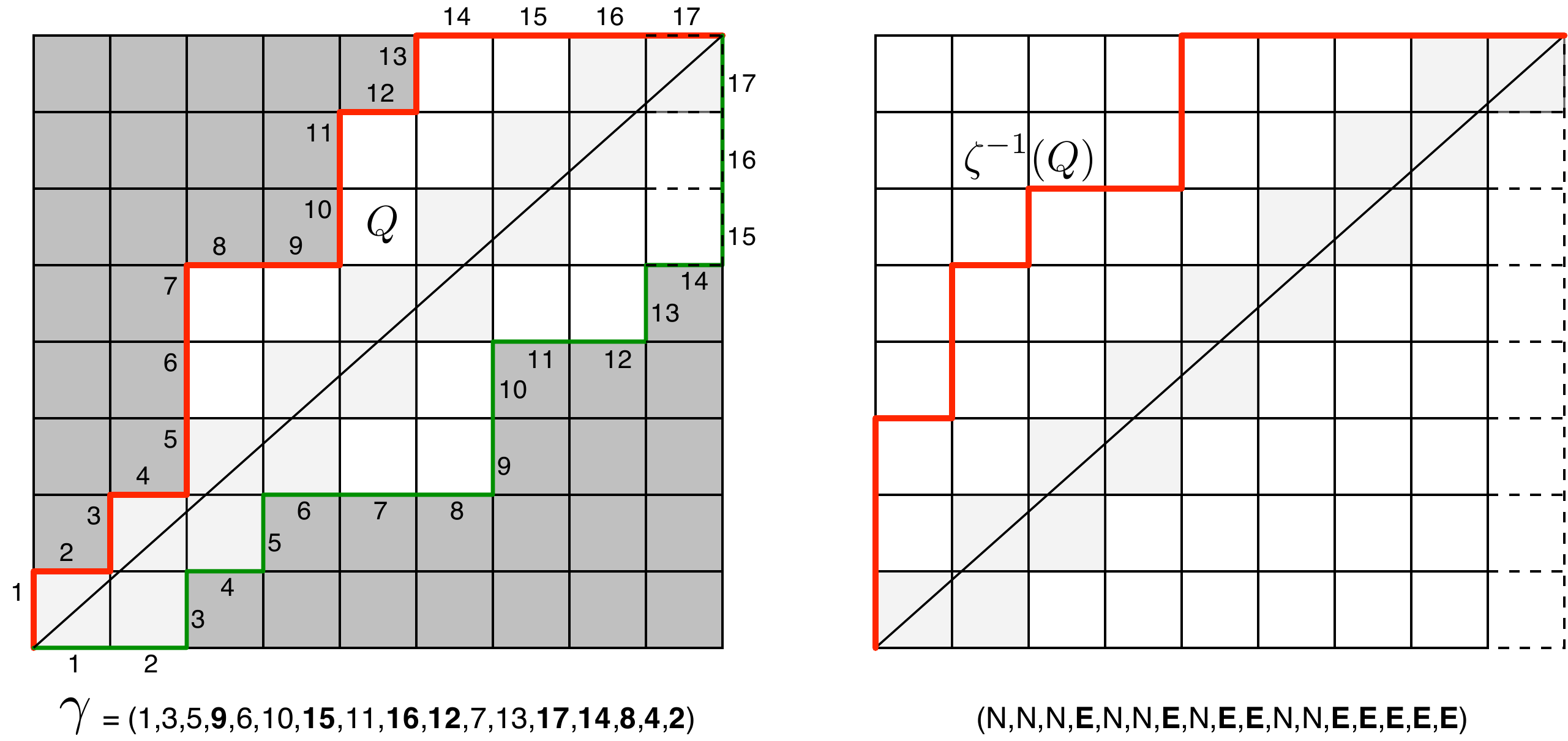}
  \end{center}
  \caption{The inverse of $\zeta$ by way of conjugate partitions.}
  \label{fig:square_zeta_inverse1}
\end{figure}

An alternative way to obtain the cycle permutation $\gamma$ directly from $\Q$ is as follows. Shade the boxes in the $n\times (n+1)$ rectangle that are crossed by the main diagonal as illustrated in Figure~\ref{fig:square_zeta_inverse2}. Move east from a vertical step labeled $i$ until the center of the first shaded box you see, and then move up until hitting an horizontal step of the path. The image $\gamma(i)$ is equal to the label of this horizontal step plus 1. In the example of Figure~\ref{fig:square_zeta_inverse2}, the path starting at the vertical step labeled 7 hits the horizontal step 12, therefore $\gamma(7)=12+1=13$. 

In order to determine $\gamma(i)$ of a label of an horizontal step, we move down until the center the last shaded box we see, and then move left until hitting a vertical step of the path. As before, $\gamma(i)$ is equal to the label of this vertical step plus 1. In the example, $\gamma(15)=10+1=11$.  The image of the label of the first horizontal step of the path is by definition equal to 1. Interpret $\gamma$ in cycle notation as $(\sigma_1,\sigma_2,\dots,\sigma_{2n+1})$ where we fix $\sigma_1=1$. 
As a direct consequence of Theorem~\ref{thm_square_inverse} we get:

\begin{theorem}\label{thm_square_inverse2}
Let $Q$ be an $(n,n+1)$-Dyck path. The inverse $\zeta^{-1}(\Q)$ is the path whose east steps correspond to the cyclic descents of the permutation~$\gamma$ when interpreted in one line notation.
\end{theorem}  

\begin{figure}
  \begin{center}
  \includegraphics[scale=0.5]{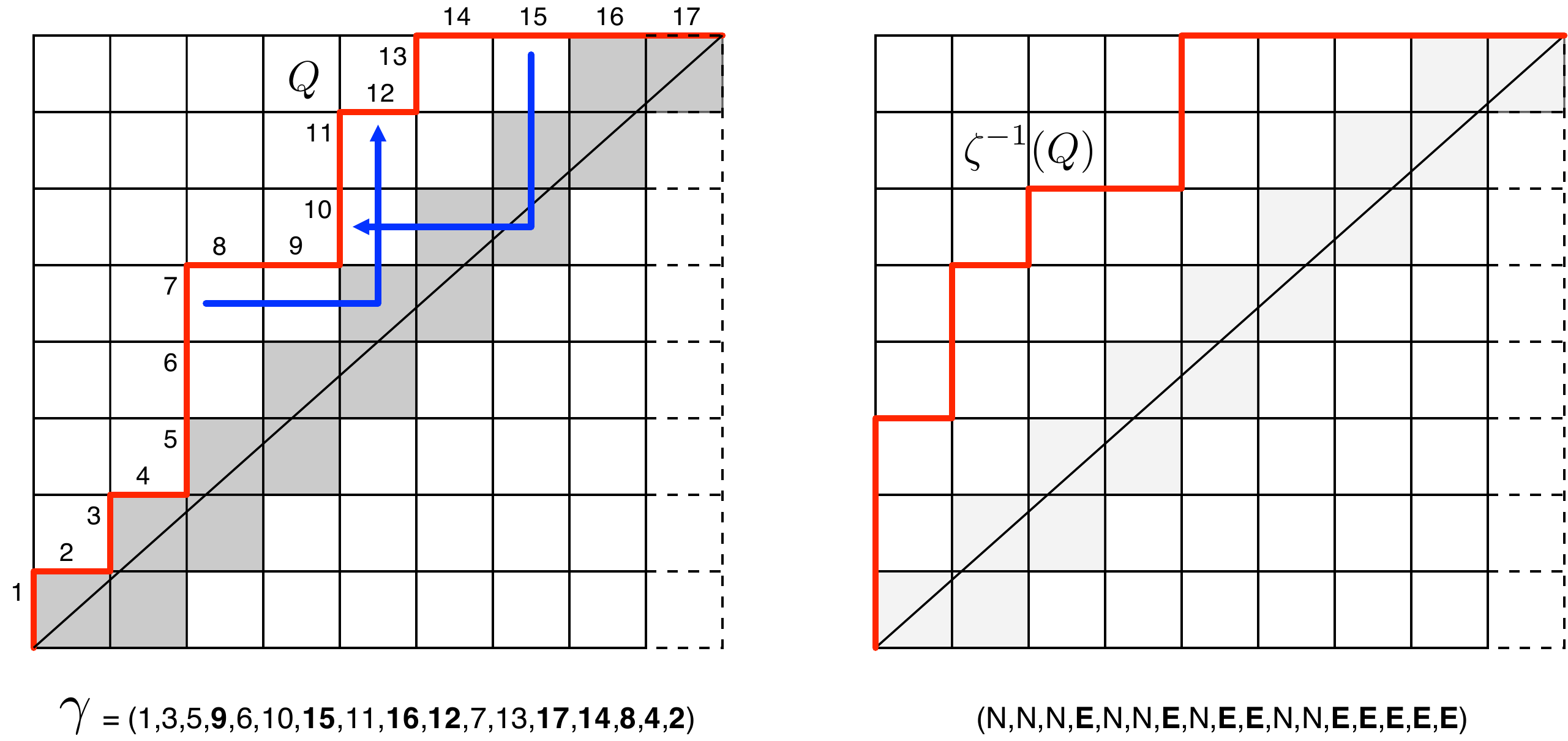}
  \end{center}
  \caption{Alternative description of the cycle permutation $\gamma$.}
  \label{fig:square_zeta_inverse2}
\end{figure}

\section{Zeta inverse and area-preserving involution for a nice family of examples}\label{sec:inductive_zeta_inverse}

In this section we present an inductive combinatorial description of the inverse of the zeta map and of the conjugate-area involution~$\chi$ for a nice family of $(a,b)$-Dyck paths. This family consists of the Dyck paths that contain the lattice point with level $1$. Such Dyck paths are obtained by concatenating two Dyck paths in the $a'\times b'$ and $a'' \times b''$ rectangles illustrated in Figure~\ref{fig:base_induction}. The sides of these two rectangles are the unique positive integers $0<a',a''<a$ and $0<b',b''<b$ such that 
\begin{align*} 
a' b - b' a &=1, \\ 
b'' a - a'' b &=1. 
\end{align*}
As a consequence, $a'$ and $b'$ are relatively prime as well as $a''$ and $b''$, allowing us to apply induction.

\begin{figure}[htbp]
\includegraphics[width=0.6\textwidth]{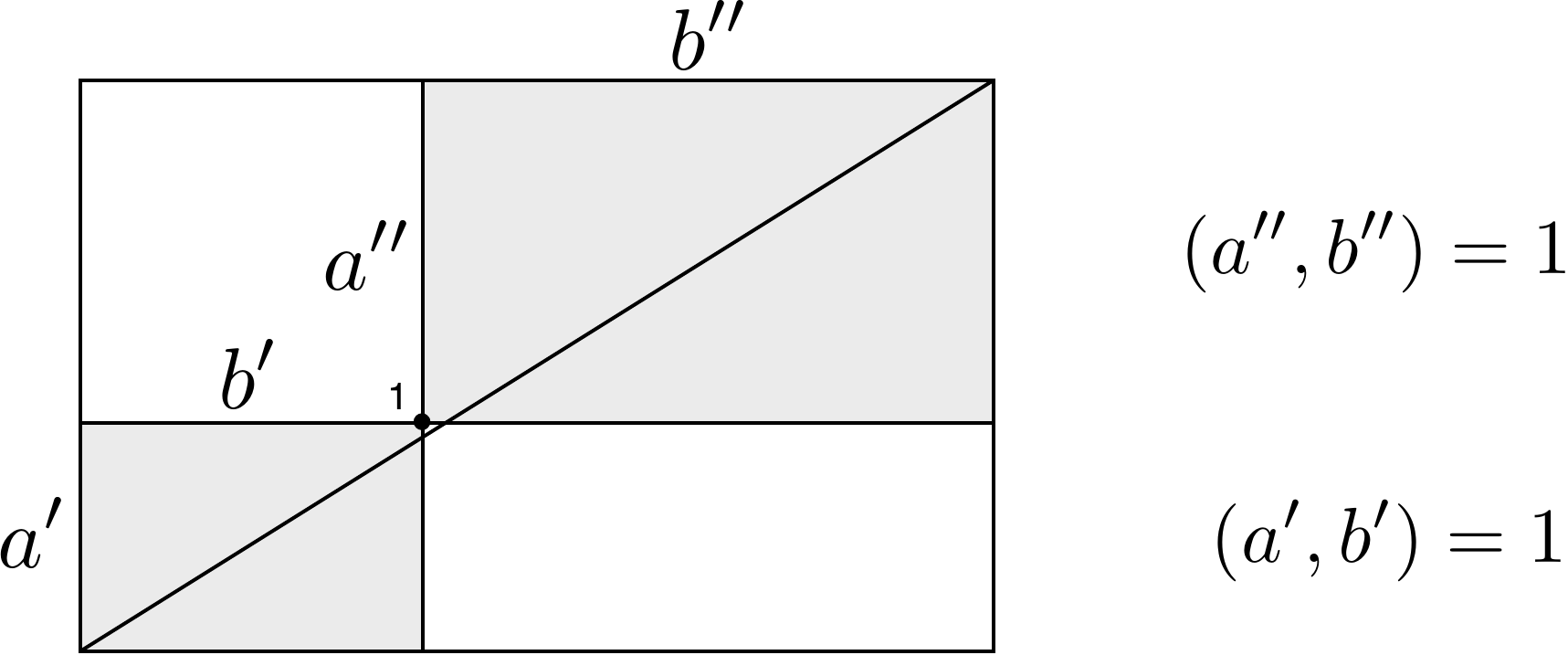}
\caption{Base induction for zeta inverse and the area-preserving involution.}
\label{fig:base_induction}
\end{figure}

\subsection{Zeta inverse}
Let $P$ be an $(a,b)$-Dyck path congaing the lattice point at level $1$, and let $P'$ and $P''$ be the two Dyck paths in the  $a'\times b'$ and $a'' \times b''$ rectangles whose concatenation is equal to~$P$.
Define the \emph{star product} of $P' \star P''$ as the path obtained by cutting~$P'$ at its highest level and infixing $P''$. This special product is illustrated in Figure~\ref{fig:star_product}. Note that the highest level of $P'$ can be equivalently obtained by sweeping the main diagonal of either the $a\times b$ rectangle or the $a'\times b'$ rectangle.

\begin{figure}[htbp]
\includegraphics[width=0.7\textwidth]{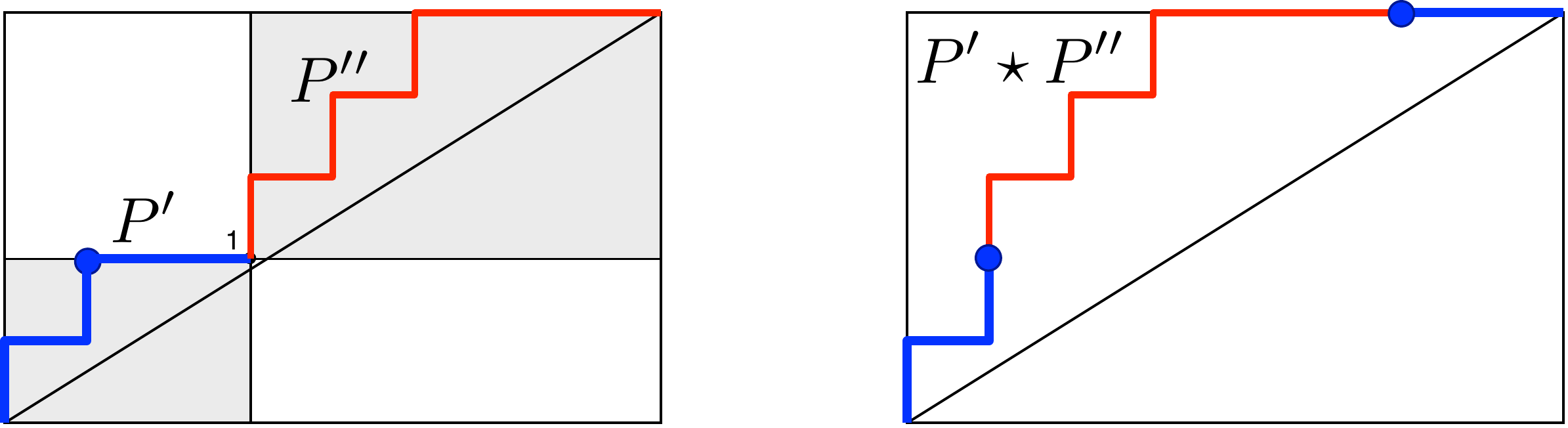}
\caption{Star product of rational Dyck paths.}
\label{fig:star_product}
\end{figure}

\begin{theorem}\label{thm:inductive_inverse}
If $Q$ is an $(a,b)$-Dyck path containing the lattice point at level 1, zeta inverse of $Q$ is equal to the star product of the zeta inverses of $Q'$ and $Q''$:
\[\zeta^{-1}(Q) = \zeta^{-1}(Q') \star \zeta^{-1}(Q''). \]
\end{theorem}

\begin{proof}
We will show that $\zeta(P'\star P'')$ is the concatenation of $\zeta(P')$ and $\zeta(P'')$, the theorem then follows by applying zeta to both sides of the equation. Since the path $P'$ is cut at its highest level, sweeping the main diagonal of the $a\times b$ rectangle crosses the levels of $P'\star P''$ corresponding to the path $P'$ first, followed by all the levels corresponding to the path $P''$. Therefore, $\zeta(P'\star P'')$ is the concatenation of $\zeta(P')$ and $\zeta(P'')$.
\end{proof}

\subsection{Area-preserving involution} The conjugate-area map of $Q$ can be obtained by induction in this case as well.

\begin{lemma}\label{lem:areadif_level}
Let $l$ be the level of a lattice point $p$ in the $a\times b$ grid. If $U_l$ is the rectangle composed by the boxes northwest of $p$ and $\tilde U_l$ is the rectangle composed by the boxes southeast of $p$, then 
\[
\area(\tilde U_l)-\area(U_l) = l.
\]  
\end{lemma}

\begin{proof}
If $p=(p_1,p_2)$, then $l = p_2b-p_1a$. Furthermore,
\[
\area(\tilde U_l) - \area(U_l)= (b-p_1)p_2 - p_1(a-p_2) = p_2b-p_1a =l.
\]
\end{proof}

\begin{theorem}\label{thm:inductive_area}
Let $Q$ is an $(a,b)$-Dyck path containing the lattice point at level 1. The bounded partition of $\chi(Q)$ is the partition whose restriction to the $a'\times b'$ and $a''\times b''$ rectangles gives the bounded partitions of $\chi(Q')$ and $\chi(Q'')$, and which contains all boxes below the main diagonal outside the two rectangles.
\end{theorem}

\begin{proof}
We first observe that $\chi(Q)$ defined this way has the same area of $Q$. This is equivalent to show that the bounded partitions of $\chi(Q)$ and $Q$ have the same area, when restricted to the complement of the $a'\times b'$ and $a''\times b''$ rectangles. These restrictions are exactly the rectangle $\tilde U_1$ after removing the box on its upper left corner, and the rectangle $U_1$. The claim then follows by Lemma~\ref{lem:areadif_level}. 

 Now, let $\gamma'$ and $\gamma''$ be the cycle permutations arising from the pairs $(Q',\chi(Q'))$ and $(Q'',\chi(Q''))$, and $\gamma$ be the cycle permutation of $(Q,\chi(Q))$.
 The cycle permutation $\gamma$ can be obtained by cutting~$\gamma'$ exactly before its highest value and putting $\gamma''$ in between, with all its values increased by $a'+b'$. In the example in Figure~\ref{fig:inductive_area} we get
\begin{align*}
 \gamma' &= ({\color{blue} 1, 3\ |\ 5, 4, 2}) \\
 \gamma'' &= ({\color{red} 1, 3, 7, 5, 8, 6, 4, 2}) \\
 \gamma &= ( {\color{blue}1, 3,} \ {\color{red}6, 8, 12, 10, 13, 11, 9, 7} \ {\color{blue} 5, 4, 2})
\end{align*}

The cyclic descents of $\gamma$ correspond exactly to the cyclic descents of $\gamma'$ and $\gamma''$. Moreover, the descent at the highest value of $\gamma'$ corresponds to the east step at the highest level of $\zeta^{-1}(Q')$. Thus, replacing cyclic descents in $\gamma$ by east steps and ascents by north steps gives rise to the start product~$\zeta^{-1}(Q')\star \zeta^{-1} (Q'')$, which is equal to $\zeta^{-1}(Q)$ by Theorem~\ref{thm:inductive_inverse}. 
\end{proof}

\begin{figure}[htbp]
\includegraphics[width=0.3\textwidth]{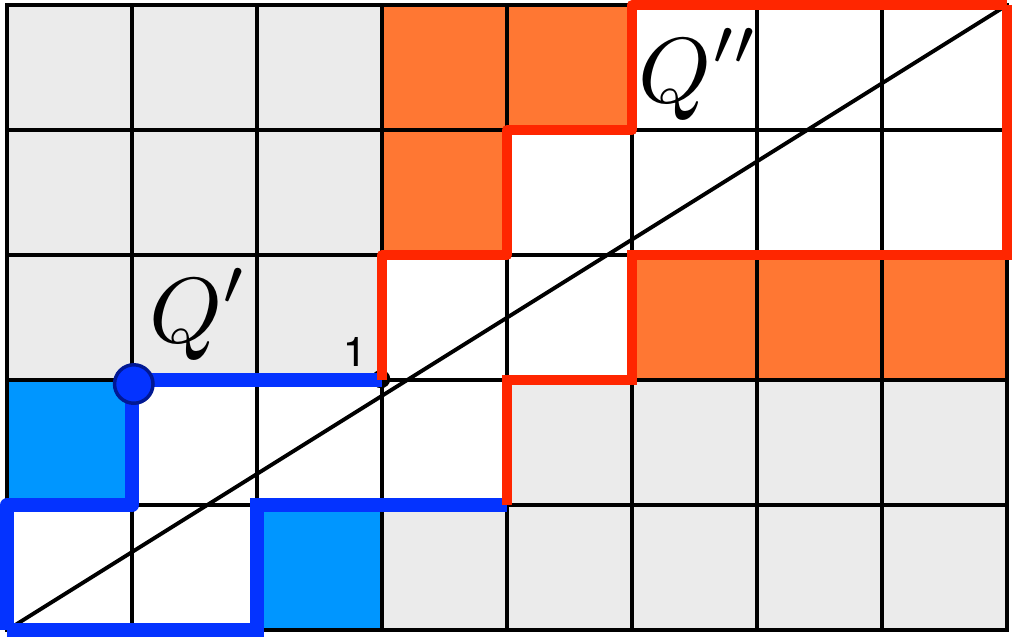}
\caption{The inductive conjugate-area map for paths containing level 1.}
\label{fig:inductive_area}
\end{figure}

\subsection{$k$th valley Dyck paths}
One interesting family of $(a,b)$-Dyck paths is the family of \emph{$k$th valley Dyck paths}, paths $Q_k$ with valleys at levels $0,1,2,\dots , k$ for some $k < a$. The area-conjugate map for these paths behaves very nice and can be described in terms of the rectangles $U_l$ and $\tilde U_l$ in Lemma~\ref{lem:areadif_level}.

For $0<l<a$, consider the collections of boxes $V_l$ and $\hat V_l$ defined by
\[
 V_l = U_l \smallsetminus \bigcup_{i=1}^{l-1} U_i, \hspace{1cm}
 \hat V_l = \hat U_l \smallsetminus \bigcup_{i=1}^{l-1} \hat U_i,
\]
where $\hat U_i$ is composed of the boxes of  $\tilde U_i$ that are below the main diagonal. Equivalently, $\hat U_i$ is the result of removing the box in the upper left corner of  $\tilde U_i$. An example is illustrated in Figure~\ref{fig:kthDyckpaths}.

\begin{lemma}
For $0<l<a$, the area of $V_l$ is equal to the area of $\hat V_l$.
\end{lemma}

\begin{proof}
Since $V_1=U_1$ and $\hat V_1=\hat U_1$, which is $\tilde U_1$ after removing one box, Lemma~\ref{lem:areadif_level} implies that~$V_1$ and $\hat V_1$ have the same area. The level 2 becomes level 1 in the smaller $a''\times b''$ rectangle, and~$V_2,\hat V_2$ are given by $U_1,\hat U_1$ in this smaller rectangle. Again, Lemma~\ref{lem:areadif_level} implies that $V_2$ and $\hat V_2$ have the same area. Continuing the same argument in the smaller rectangles that appear in the process finishes the proof.
\end{proof}

\begin{figure}[htbp]
\includegraphics[width=0.45\textwidth]{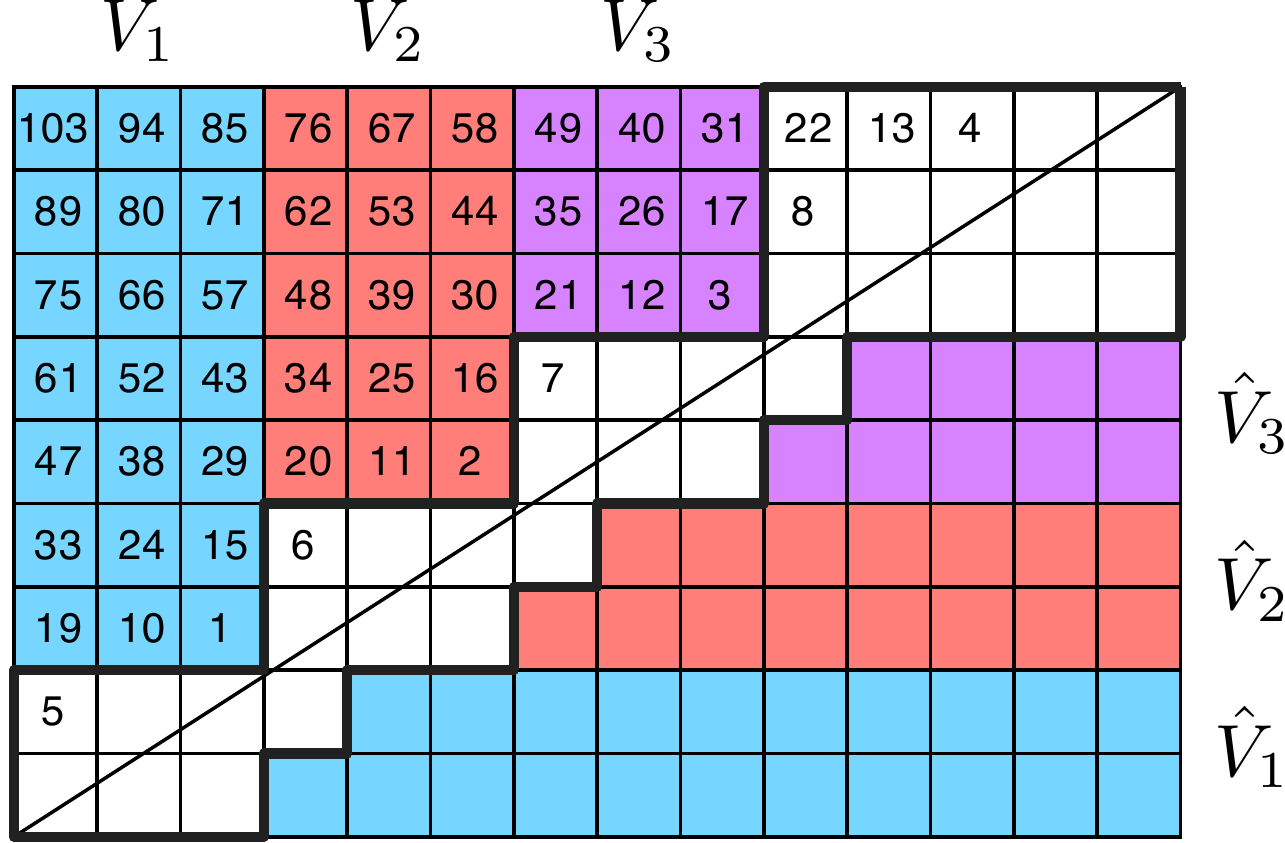}
\caption{Example of the conjugate-area map for $k$th valley Dyck paths for $k=3$. The area of $V_i$ is equal to the area $\hat V_i$.}
\label{fig:kthDyckpaths}
\end{figure}

Note that the bounded partition of $Q_k$ is the (disjoint) union of $V_1,\dots,V_k$. 

\begin{proposition}
The bounded partition of $\chi(Q_k)$ is the (disjoint) union of $\hat V_1,\dots,\hat V_k$.
\end{proposition}

\begin{proof}
Note that the restriction of $Q_k$ to the $a'\times b'$ and $a''\times b''$ rectangles gives two smaller $k$th valley Dyck paths $Q'_{k'}$ and $Q''_{k''}$. The result then follows directly from Theorem~\ref{thm:inductive_area} by induction on~$k$.  
\end{proof}

Figure~\ref{fig:kthDyckpaths_inverse} illustrates an example of the inverse of the zeta map for $k$th valley Dyck paths obtained by applying Theorem~\ref{thm:inverse}.

\begin{figure}[htbp]
\includegraphics[width=0.9\textwidth]{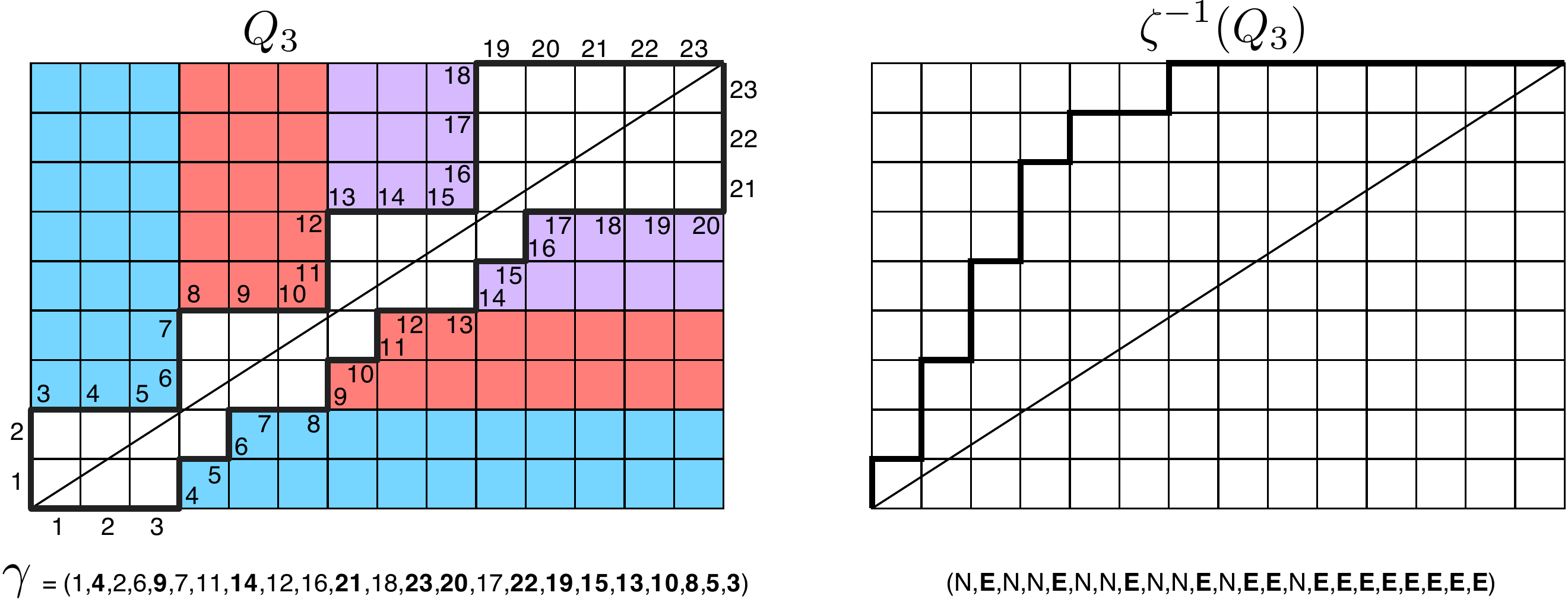}
\caption{Example of the inverse of the zeta map for $k$th valley Dyck paths for $k=3$.}
\label{fig:kthDyckpaths_inverse}
\end{figure}

\section{The delta statistic and initial bounce paths}\label{sec:9}

We have tried a number of approaches for showing that the zeta map is a bijection. Our last approach uses a delta statistic which can be estimated by means of an initial bounce path.

\begin{definition}
Define the \emph{delta statistic} $\delta(\P)$ to be the number of levels $l_i<a+b$ along $\P$. 
\end{definition}

Geometrically, $\delta(\P)$ counts the number of lattice points in $P$ that belong to the diagonal path (closest to the diagonal). Surprisingly, this is sufficient to construct the inverse of the zeta map.

\begin{theorem}\label{cor:delta_inverse}
If $\delta(\P)$ is determined uniquely by $\zeta(\P)$ for all $\P$, then the map $\zeta$ is invertible.
In such case, the inverse path $\P$ is determined from $\gamma(\P)$ as in Proposition~\ref{rem:zetainverse_delta}. 
\end{theorem}

\subsection{Box math and the inclusion poset of rational Dyck paths}\label{sec:82}\
Our approach for proving Theorem~\ref{cor:delta_inverse} 
relies on careful analysis of the poset structure on the set of rational Dyck paths under the usual inclusion relation: we say $\P<\Q$ if the path $\P$ is weakly below the path $\Q$, or, equivalently, if the set of positive hooks of $\P$ is contained in the set of positive hooks of $\Q$.  This poset is graded by the area statistic, with covering relation given by adding a single box. 

\begin{definition}
The \emph{maximal level} $m$ of a path $\P$ is the largest level appearing in the reading word of $L(\P)$.  Likewise, the \emph{maximal box} is the box under the peak of $\P$ labeled by the maximal level~$m$.  For any path with area greater than $0$, we define the \emph{predecessor} of $\P$ as the path obtained by removing the box under the peak of $\P$ farthest from the diagonal. This replaces the maximal level $m$ with~$m-a-b$ in $L(P)$.  
\end{definition}

\begin{lemma}
Suppose that $\P$ is an $(a,b)$-Dyck path with predecessor $\P'$. We have~$\sl(\P')<\sl(\P)$.
\end{lemma}
\begin{proof}
Since~$\P'$ is obtained by removing the maximal box of~$\P$, the laser filling of $\P'$ is equal to the laser filling of $\P$ when removing the laser filling of its maximal box. Since skew length is equal to the sum of the entries in the laser filling, the result follows.
\end{proof}

Since every path has a unique maximal hook, we induce a spanning tree $T$ in the Hasse diagram of $\DD$ with the property that if $\P'<\P$ in $T$, then $\sl(\P')<\sl(\P)$.
We can also precisely describe the combinatorial effect of removing the maximal box from $\P$ on~$\zeta(\P)$.

\begin{proposition}
All of the following operations are equivalent ways to remove the maximal box:
\begin{enumerate}
\item Remove the box whose associated hook length is greatest (furthest box from the diagonal).
\item Remove the longest row from $\core(\P)$.
\item In the reading word $(l_1, l_2, \ldots, l_{a+b})$ of $\P$, reduce the maximal level $m$ by $(a+b)$, leaving all other levels unchanged.
\item In the standardization $\sigma(\P)$, let $\alpha$ be the number of levels of~$\P$ greater than $m-a-b$ excluding $m$.  Replace the entry $(a+b)$ in $\sigma$ with $a+b-\alpha$, and increase all entries greater than or equal to $a+b-\alpha$ by one.  Equivalently, multiply $\sigma(\P)$ on the left by the cycle permutation $\rho_{a+b-\alpha,a+b}$ with cycle notation $(a+b-\alpha, a+b-\alpha+1, \ldots, a+b)$.
\item Conjugate the permutation $\gamma(\P)$ by the cycle $\rho_{a+b-\alpha,a+b}$ to obtain:
\[\rho_{a+b-\alpha,a+b} \gamma(\P)  \rho_{a+b-\alpha,a+b}^{-1}.\]
\end{enumerate}
\end{proposition}

\begin{proof}
The second operation follows directly from the definition of $\core(\P)$.  The third method is clear from the effect on the labels $l_i$ of applying the first method.  The fourth item follows from third, and the fifth item follows from the effect of conjugation on the cycle notation of a permutation.
\end{proof}

We can thus try to understand the structure of the tree $T$ by understanding certain conjugations of the permutation $\gamma(\P)$.
We can observe that adding the box at the label $l_1=0$ is equivalent to removing the maximal box from $\P^c$.  This reduces all labels $l_i$ by $a+b$ except for the label $l_1=0$, which remains the same.  As a result, the relative value of the label $0$ is increased from $1$ to 
the number $\delta(P)$ of labels $l_i<a+b$, 
while all other labels $\leq \delta(P)$ are reduced by one.  

\begin{proposition}\label{prop:predecessor1}
Let $\P'$ be the conjugate of the path obtained by removing the maximal box from~$\P^c$. The permutation~$\gamma(\P')$ is the conjugate 
\[
\gamma(\P') = \rho_{1,\delta(\P)}^{-1} \gamma(\P) \rho_{1,\delta(\P)}.
\]
\end{proposition}

The action of removing the maximal box from~$\P^c$ on~$Q=\zeta(\P)$ is also completely determined by~$\delta(\P)$. For simplicity, we call the path $\Q'=\zeta(P')$ the \emph{$\zeta$-predecessor} of $Q$. 
\begin{proposition}\label{prop:predecessor2}
The $\zeta$-predecessor of $\Q=\zeta(P)$ is completely determined by $\delta=\delta(P)$ as follows:
\begin{enumerate}[1.]
\item All the steps in $\Q$ after the first $\delta$ steps remain unchanged.
\item The first $\delta$ steps are rotated as follows:
\begin{enumerate}[(a)]
\item the first east step that appears is changed to a north step,
\item the first north step is changed to an east step,
\item the first $\delta$ steps are rotated once (rotating the first step to the end of the first delta steps). 
\end{enumerate}
\end{enumerate}
\end{proposition}

\begin{example}
An example of this procedure is illustrated in Figure~\ref{fig:delta_Q_predecessor} for the path~$\Q=\zeta(\P)$ associated to our running example path $\P$. The number of levels of~$\P$ smaller than $a+b$ is~$\delta(P)=5$. The cycle permutation $\gamma'$ is obtained from $\gamma$ by rotating the labels $1,\dots,5$. 
\end{example}

\begin{figure}[htb]
  \includegraphics[width=0.8\textwidth]{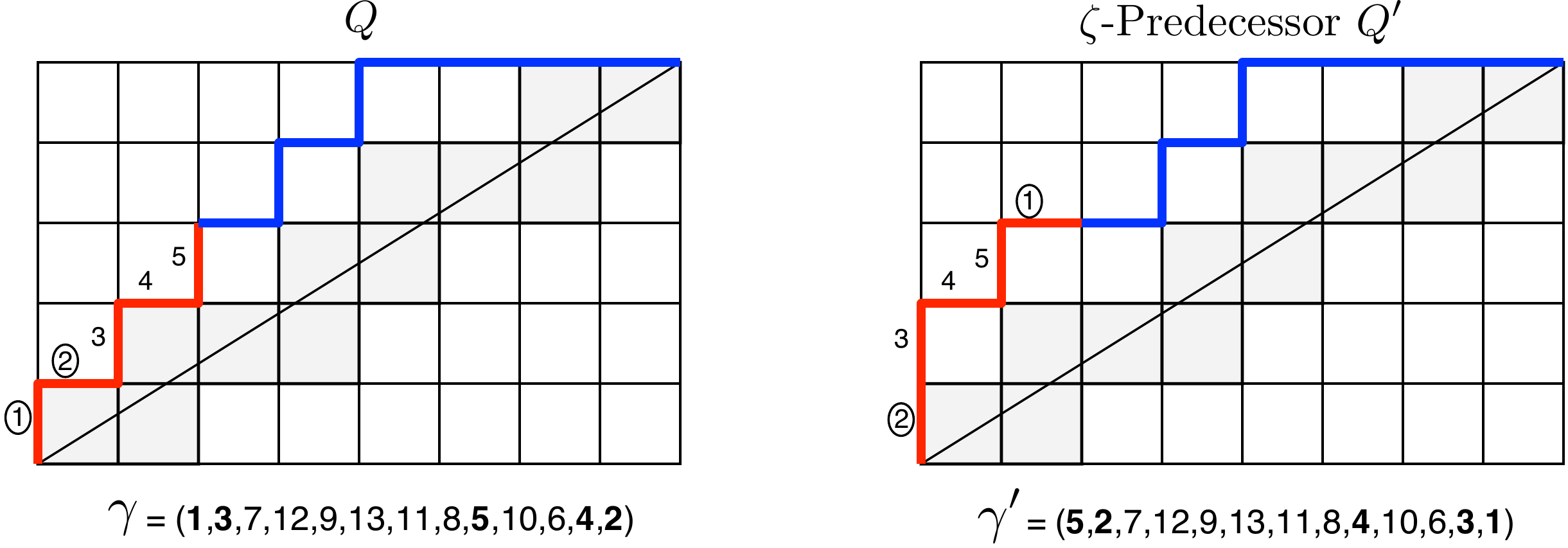} 
  \caption{Determining the $\zeta$-predecessor of a path $\Q=\zeta(\P)$ with $\delta(\P)$.}
  \label{fig:delta_Q_predecessor}
\end{figure}

\begin{proof}
Adding a box at the level $\ell_1=0$ of $P$ increases this level by $a+b$ and all other levels remain unchanged. The level $\ell_1=0$ is transformed from a north level in $P$ to an east level in $P'$, while the east level $\ell_i=a$ is transformed to a north level. Observing that the levels $\ell_1$ and $\ell_i$ correspond to the first north and first east steps in $Q$ respectively, and that the relative order of the levels less than $a+b$ is rotated once, one concludes the result. 
\end{proof}

The two formulations in Proposition~\ref{prop:predecessor1} and Proposition~\ref{prop:predecessor2} have the advantage that we do not actually need to know the value of the maximal label $m$, and reduces the problem of showing that $\zeta$ is a bijection to computing a single statistic.  To wit, if~$\delta(\P)$ can be directly computed from~$\Q=\zeta(\P)$, then we can  obtain the $\zeta$-predecessor of $\Q$, and repeat until we arrive at the diagonal path~$\P_0$. This would completely determine $\P$ from $\zeta(\P)$. 

\begin{proposition}\label{rem:zetainverse_delta}
Let $Q=\zeta(P)$ and $\Q=\Q_1,\dots,\Q_l$ be the list of $\zeta$-predecessors of $\Q$ with $\Q_l$ being the final path containing all boxes above the main diagonal. If $\delta_i$ is the $\delta$-statistic corresponding to~$\Q_i$, then the permutation $\gamma(P)$ is determined by 
\[
\gamma(P) = \rho\  \gamma_0\ \rho^{-1},
\]
where $\rho= \rho_{1,\delta_1}\dots \rho_{1,\delta_{l-1}}$, the permutation $\rho_{1,i}$ has cycle notation $(1, \dots , i)$, and $\gamma_0=\gamma(\P_0)$. 
The east steps of the path $\P$ are encoded by the cyclic descents of $\gamma$ when considered in one-line notation, as described in Section~\ref{sec:zetaeta}. 
\end{proposition}

\begin{remark}
The $\delta$ statistic has also been considered shortly after this paper by Xin in~\cite{xin_efficient_2015}. This statistic, and another related statistic called ``key", is used by the author to present a search algorithm for inverting the zeta map. This algorithm shows an alternative proof that the zeta map is a bijection in the cases~$(a,ak\pm 1)$, by giving a recursive construction of~$\zeta^{-1}(\Q)$. However, the general case remains open. The results of Xin in~\cite{xin_efficient_2015} are very similar to those presented in this section, and also use the operations of removing the maximal box in $\P$ and $\P^c$. Corollary~\ref{cor:delta_inverse} should be compared with~\cite[Corollary~19]{xin_efficient_2015}. We also present estimates for the~$\delta$ statistic and a precise formula in the Fuss-Catalan case $(a,ak+1)$ in Proposition~\ref{prop:funnydelta} and Corollary~\ref{cor:delta_Fuss}, which should be compared with~\cite[Theorem~16]{xin_efficient_2015}. The estimates we present in Proposition~\ref{prop:funnydelta} determine the number of children of the nodes in the search tree in the ``ReciPhi algorithm" in~\cite[Section~5]{xin_efficient_2015}.
Our algorithm for describing the inverse of zeta in the Fuss-Catalan case $(a,ak+1)$ 
should be compared with the ReciPhi algorithm for the Fuss-Catalan case in~\cite[Section~5]{xin_efficient_2015}.
\end{remark}

\subsection{Initial part of a rational bounce path}
The zeta map has been shown to be a bijection in the special cases $(a,am\pm 1)$ by way of a ``bounce path'' by which zeta inverse could be computed~\cite{Loehr,GMII}. However, constructing such a bounce path for the general~$(a,b)$ case remains elusive. In this section, we construct the initial part of a rational bounce path and show its relation to the $\delta$ statistic. In particular, we explicitly compute $\delta$ in the Fuss-Catalan case~$(a,ak+1)$.

\begin{definition}
Let $\Q$ be an $(a,ak+r)$-Dyck path with $0\leq k$ and $0<r<a$. The rational \emph{initial bounce path} of $\Q$ consists of a sequence of alternating $k+1$ \emph{vertical moves} and~$k$ \emph{horizontal moves}. We begin at $(0,0)$ with a vertical move followed by a horizontal move, and continue until eventually finish with the $(k+1)$th vertical move.  
Let $v_1,\dots,v_{k+1}$ denote the lengths of the successive vertical moves and $h_1,\dots, h_k$ denote the lengths of the successive horizontal moves. These lengths are determined as follows.

We start from $(0,0)$ and move north $v_1$ steps until reaching an east step of $\Q$. Next, move~$h_1=v_1$ steps east. Next, move north $v_2$ steps from the current position until reaching an east step of $\Q$.  Next, move $h_2=v_1+v_2$ steps east. In general, we move north $v_i$ steps from the current position until reaching an east step of the path, and then move east $e_i=v_1+\dots +v_i$ steps. This is done until obtaining the last vertical move $v_{k+1}$. 
\end{definition}

\begin{remark}
The definition of the initial bounce path is exactly the same as an initial part of the bounce path in the Fuss-Catalan case~\cite{Loehr,GMII}. The description of this initial part remains the same for the general $(a,b)$-case but we still do not know how to extend it to a complete bounce path in general. 
\end{remark}

It turns out that the initial bounce path is closely related to the $\delta$ statistic. Denote by $|v|=v_1+\dots +v_{k+1}$ and by $|h|=h_1+\dots +h_k$. In all the results of this section we always assume $0\leq k$ and $0<r<a$. 

\begin{proposition}\label{prop:funnydelta}
Let $\Q=\zeta(P)$ be an $(a,ak+r)$-Dyck path and $\widetilde \delta(\P) \leq \delta(\P)$ be the number of levels in~$\P$ that are less than or equal to $a(k+1)$. The two following equations hold:
\begin{equation}\label{eq:bounce1}
\widetilde \delta(\P) = |v|+|h|+1,
\end{equation}
\begin{equation}\label{eq:bounce2}
|v|+|h|+1 \leq \delta(\P) \leq |v|+|h|+r.
\end{equation}
\end{proposition}

This proposition will follow from the following lemma. Note that every such a path $\P$ contains the east levels $a,2a,\dots,(k+1)a$ at the end of the path. Moreover,

\begin{lemma}
The east steps of $\Q=\zeta(P)$ that are reached by the vertical moves $v_1,\dots,v_{k+1}$ of its initial bounce path correspond to the east levels $a,2a,\dots,(k+1)a$ of $P$.  
\end{lemma}

\begin{proof}
Denote by $A_0=\{0,1,\dots , a-1\}$ the set of natural numbers between $0$ and $a-1$, and let~$A_i=A_0+ia$ be the translation of $A_0$ by $ia$. Note that the number of boxes above the main diagonal that are directly on the right of a north step with north level in $A_i$ is exactly equal to $i$. So, these sets can be used to encode the ``area-vector" of a Dyck path.

As in the known Fuss-Catalan bounce path description, we will show that the vertical steps of~$\Q$ that are directly on the left of the $v_i$ vertical move contribute area $i-1$ in $\P$. More precisely, we will show:

\begin{enumerate}[1.]
\item The vertical steps of~$\Q$ that are directly on the left of the $v_i$ vertical move correspond to north levels in $\P$ that belong to $A_{i-1}$.
\item The horizontal steps of~$\Q$ that are directly above of the $h_i$ horizontal move correspond to east levels in $\P$ that belong to $A_i$.
\end{enumerate}

Denote by $N_i$ the set of north levels in $\P$ that belong to $A_i$, for $i=0,\dots ,k$. Similarly, denote by $E_i$ the set of east levels in $\P$ that belong to $A_i$, for $i=1,\dots, k$. We first show that $E_i$ can be obtained as the disjoint union 
\begin{equation}\label{eq:ENcorrespondence}
E_i = \bigcup_{j=0}^{i-1} \left(N_j + (i-j)a\right). 
\end{equation}

For this, note that the first north level in $\P$ that appears after an east level $e\in E_i$ should be a north level $n\in N_j$ for some $j \in \{0,\dots i-1\}$, and that $e=n+(i-j)a$. Reciprocally, every north level $n\in N_j$ for some $j \in \{0,\dots i-1\}$ forces the east levels $e_{j+1},e_{j+2}\dots, e_{k}$ to appear as east levels in $\P$, where $e_{j+l}=n+l a$ (indeed, $e_{j+l}\in A_{j+l}$ which means that $e_{k-1}<ak<b=ak+r$. Then, all the lattice points one step below $e_{j+1},e_{j+2}\dots, e_{k-1}$ are below the main diagonal). Thus, the north level $n\in N_j$ contributes with exactly one east level $e=n+(i-j)a$ in $E_i$. 

Items 1 and 2 above are now equivalent to prove that $v_i=|N_{i-1}|$ and $h_i=|E_i|$.  Equation~\ref{eq:ENcorrespondence} implies that 
$|E_i|=|N_0|+\dots +|N_{i-1}|.$
Since $h_i=v_1+\dots+v_i$, it suffices to prove $v_i=|N_{i-1}|$. 
Note that $v_1$ is clearly the number of elements in $N_0$, since the smallest east level of $\P$ (which corresponds to the first east step in $\Q$) is equal to $a$. Moving $h_1=v_1=|E_1|$ steps horizontally covers all the east steps of $\Q$ corresponding to the east levels in $\P$ that belong to $A_1$. Moving up $v_2$ units from the current position hits the path at the east step corresponding to the first east level of $\P$ that belongs to $A_2$. This east level is exactly equal to $2a$, and all the north steps on the left of $v_2$ in $Q$ correspond to the north levels in $\P$ that belong to $A_1$, that is $v_2=|N_1|$. In general, $\P$ contains all east levels $a,2a,\dots,(k+1)a$. Therefore, the initial value of $E_i$ is equal to $ia$ and is smaller than all values in $N_i$. As a consequence the vertical move $v_i$ of the bounce path hits the path $\Q$ precisely at the east step corresponding to the level $ia$ of $\P$ as desired, and $v_i=|N_{i-1}|$. This finishes the proof of the proposition and the proof of items 1 and 2.
\end{proof}

\begin{proof}[Proof of Proposition~\ref{prop:funnydelta}]
The east step of $\Q$ that is reached by the vertical move~$v_{k+1}$ corresponds to the east level $(k+1)a$ in $\P$. So, $\widetilde \delta(\P)$ is equal to the position of this east step in $\Q$, which is equal to $|v|+|h|+1$. Since $a+b=(k+1)a+r$ and $a+b$ never appears as a level in $\P$, then~$\delta\leq \widetilde \delta + r-1 = |v|+|h|+r$.
\end{proof}

Replacing $r=1$ in Equation~\eqref{eq:bounce2}, we obtain.
\begin{corollary}\label{cor:delta_Fuss}
In the Fuss-Catalan case $(a,ak+1)$, the statistic $\delta(\P)$ is determined by the initial bounce path of $\Q=\zeta(\P)$ by
\begin{equation}
\delta(\P) = |v|+|h|+1.
\end{equation}
\end{corollary}

As a consequence of Corollary~\ref{cor:delta_inverse} and Corollary~\ref{cor:delta_Fuss} we obtain an alternative proof that the zeta map is a bijection in the Fuss-Catalan case~$(a,ak+1)$, as previously shown by Loehr in~\cite{Loehr}. 

\begin{corollary}\label{cor:zetabijection_Fuss}
The zeta map is a bijection in the Fuss-Catalan case $(a,ak+1)$.
\end{corollary}

\section*{Acknowledgements}
\label{sec.ack}

Some of the results in this work are partially the result of working sessions at the Algebraic Combinatorics Seminar at the Fields Institute with the active participation of Farid Aliniaeifard, Nantel Bergeron, Eric Ens, Sirous Homayouni, Sahar Jamali, Shu Xiao Li, Trueman MacHenry, and Mike Zabrocki. 

We thank two anonymous referees for their comments and for pointing us to some references of previously known results.

We especially thank Rishi Nath for discussions involving core partitions that helped simplify arguments in Section~\ref{sec:skew_length}, Drew Armstrong for sharing his knowledge of the history of this subject, and Greg Warrington for pointing us the references about the $\dinv$ statistic in Remark~\ref{rem:dinv}.  We are also grateful to York University for hosting a visit of the third author.

\bibliographystyle{amsalpha}
\bibliography{abCores}

\end{document}